\documentclass[11pt,letterpaper]{amsart}

\usepackage[margin=1.08in]{geometry}
\usepackage{amsmath,amssymb,amsthm,mathtools}
\usepackage{microtype}
\usepackage[colorlinks=true,linkcolor=blue,citecolor=blue,urlcolor=blue]{hyperref}
\usepackage[nameinlink,capitalize]{cleveref}
\usepackage{lea-tags}

\numberwithin{equation}{section}

\DeclareMathOperator{\tr}{tr}
\DeclareMathOperator{\dist}{dist}
\DeclareMathOperator{\osc}{osc}
\DeclareMathOperator{\Lip}{Lip}

\newcommand{\TT}{\mathcal{T}}
\newcommand{\GG}{\mathcal{G}}

\newcommand{\Sn}{\mathcal S^n}
\newcommand{\eps}{\varepsilon}
\newcommand{\norm}[1]{\left\lVert #1\right\rVert}

\newenvironment{leatargetblock}{}{}
\newenvironment{leatargetcontinuation}{\noindent}{}

\newtheorem{theorem}{Theorem}[section]
\newtheorem{proposition}[theorem]{Proposition}
\newtheorem{lemma}[theorem]{Lemma}
\newtheorem{corollary}[theorem]{Corollary}

\theoremstyle{definition}
\newtheorem{definition}[theorem]{Definition}
\theoremstyle{remark}
\newtheorem{remark}[theorem]{Remark}

\AddToHook{env/proposition/begin}{\crefalias{theorem}{proposition}}
\AddToHook{env/lemma/begin}{\crefalias{theorem}{lemma}}
\AddToHook{env/corollary/begin}{\crefalias{theorem}{corollary}}
\AddToHook{env/claim/begin}{\crefalias{theorem}{claim}}
\AddToHook{env/definition/begin}{\crefalias{theorem}{definition}}
\AddToHook{env/remark/begin}{\crefalias{theorem}{remark}}

\crefname{theorem}{theorem}{theorems}
\Crefname{theorem}{Theorem}{Theorems}
\crefname{proposition}{proposition}{propositions}
\Crefname{proposition}{Proposition}{Propositions}
\crefname{lemma}{lemma}{lemmas}
\Crefname{lemma}{Lemma}{Lemmas}
\crefname{corollary}{corollary}{corollaries}
\Crefname{corollary}{Corollary}{Corollaries}
\crefname{claim}{claim}{claims}
\Crefname{claim}{Claim}{Claims}
\crefname{definition}{definition}{definitions}
\Crefname{definition}{Definition}{Definitions}
\crefname{remark}{remark}{remarks}
\Crefname{remark}{Remark}{Remarks}

\title[Lipschitz regularity]{Uniform Lipschitz regularity for two-phase singularly perturbed fully nonlinear elliptic equations}

\author[T. M. Nascimento]{Thialita M. Nascimento}
\address{Department of Mathematics, Universidade Federal da Para\'iba, Jo\~ao Pessoa, PB, Brazil}
\email{nascimento@mat.ufpb.br}

\author[A. Sobral]{Aelson Sobral}
\address{Applied Mathematics and Computational Sciences, King Abdullah University of Science and Technology, Thuwal, Saudi Arabia}
\email{aelson.sobral@kaust.edu.sa}

\author[E. V. Teixeira]{Eduardo V. Teixeira}
\address{Department of Mathematics, Oklahoma State University, Stillwater, OK, USA}
\email{eduardo.teixeira@okstate.edu}


\begin{document}

\begin{abstract}
We study sign-changing viscosity solutions of the singularly perturbed fully
nonlinear equation
$$
F(D^2u_\varepsilon)
=
\frac{\alpha}{\varepsilon}
\beta\left(\frac{u_\varepsilon}{\varepsilon}\right)
\qquad\text{in }B_1\subset\mathbb R^n,
$$
where $F$ is uniformly elliptic and
$\beta\in C_c (-1,1)$ is nonnegative. We prove the scale-sharp
estimate
$$
\|\nabla u_\varepsilon\|_{L^\infty(B_{1/2})}
\leq
C\left(
\|u_\varepsilon\|_{L^\infty(B_1)}+\sqrt\alpha
\right),
$$
with $C$ depending only on the dimension, the ellipticity constants, and
$\beta$, and independent of $\varepsilon$ and $\alpha$. This removes a
longstanding compactness obstruction in the analysis of fully nonlinear
two-phase singular perturbations. The difficulty is structural: at positive
$\varepsilon$ there is neither a free boundary nor a prescribed transmission
law, while the general fully nonlinear setting provides no monotonicity
formula capable of controlling the interaction of the two phases. The proof develops a diffuse counterpart of the De Silva--Savin
decay-versus-Lipschitz alternative. Exact planar transitions furnish the local
comparison geometry, and curved-test compactness carries this geometry across
collapsing reaction layers. An intrinsic transition-region estimate reduces
the problem to linear growth from buffered level boundaries. The resulting
dyadic continuation is closed by a large-slope stopping argument: bounded
accumulated slopes yield the desired growth directly, whereas unbounded slopes
force the effective reaction strength to vanish after normalization and lead
to a contradiction. The estimate is quantitatively optimal and supplies the
scale-invariant compactness framework required for the subsequent
sharp-interface analysis.
\end{abstract}

\maketitle

\section{Introduction}\label{sec:intro}

\subsection{Diffuse interfaces and the two-phase obstruction}

Uniform regularity is the mechanism by which a diffuse singular perturbation
becomes a free-boundary problem. For each fixed $\eps>0$, the transition has a
small but finite thickness; in combustion models, for example, $\eps$ may be
viewed as the reciprocal of a normalized activation energy, and the reaction
is distributed across a thin region rather than concentrated on an idealized
interface. Estimates independent of $\eps$ serve two purposes. They describe
the diffuse model without resolving its microscopic scale, and they place its
sharp-interface limits in a compact class. We refer to
\cite{WilliamsCombustion,BuckmasterLudford,
BerestyckiLarrouturouLions,CaffarelliVazquez} for classical background from
combustion and reaction--diffusion theory.

The modern study of uniform estimates for such regularizations begins with
Berestycki, Caffarelli, and Nirenberg \cite{BCN}; see also
\cite{CaffarelliUniform,Weiss} and
\cite[Section~1.2]{CaffarelliSalsa}. In the one-phase regime, positivity
provides one-sided control through Harnack-type inequalities.
This structure has led to uniform regularity results for elliptic, parabolic,
degenerate, and nonlocal problems; see, among others,
\cite{CaffarelliVazquez,CaffarelliKenig,BerestyckiHamel,
CaffarelliLeeMellet,FernandezBonderWolanski,
DanielliPetrosyanShahgholian,LedermanWolanskiNonlocal,
TeixeiraVariational,AraujoRicarteTeixeira,
RicarteTeymurazyanUrbano,daSilvaViloria}. For fully nonlinear elliptic
equations, Ricarte and Teixeira \cite{RicarteTeixeira} established the optimal
uniform Lipschitz estimate for nonnegative solutions and developed the
corresponding asymptotic free-boundary theory.

The sign-changing problem is qualitatively different. Both phases remain
active near the reaction layer, and positivity no longer supplies a preferred
side from which to control the solution. For linear, divergence-form variational
equations, the interaction of the phases can be governed by
Alt--Caffarelli--Friedman monotonicity and its variants; see
\cite{AltCaffarelliFriedman,LedermanWolanskiTwoPhase,
LedermanWolanskiForcing,CaffarelliJerisonKenig,CaffarelliLedermanWolanskiI,CaffarelliLedermanWolanskiII}. In particular,
\cite[Section~5]{CaffarelliJerisonKenig} obtains uniform gradient bounds for
semilinear two-phase regularizations of the Laplacian through an almost
monotonicity formula. Extensions to certain variable-coefficient settings are
available under additional structural assumptions
\cite{TeixeiraZhangElliptic,TeixeiraZhangParabolic,
MatevosyanPetrosyan}. These arguments use integration by parts, quadratic
energies, or distributional inequalities for the positive and negative
truncations. None of these mechanisms survives for a general fully nonlinear
operator.

There is a second, logically distinct obstruction. Sharp-interface theories
begin with a free boundary and a transmission condition. In Caffarelli's
foundational trilogy
\cite{CaffarelliI,CaffarelliII,CaffarelliIII}, and in the subsequent
fully nonlinear theories developed by Wang \cite{WangI,WangII}, Feldman
\cite{FeldmanFullyNonlinear}, Ferrari \cite{FerrariFullyNonlinear}, and
De Silva, Ferrari, and Salsa
\cite{DeSilvaFerrariSalsa,DeSilvaFerrariSalsaFullyNonlinear}, the
transmission condition is part of the formulation of the problem. More precisely, the solution is
required to satisfy, in the viscosity sense, a condition of the form
$$
u_\nu^+
=
G(u_\nu^-,x,\nu)
$$
on the free boundary. This condition couples the phases, restricts the
admissible blow-ups, and drives both the Lipschitz theory and the subsequent
improvement of flatness. A diffuse approximation has no interface at positive
$\eps$ and no transmission law available as an input. The interface and the
relation between its emerging one-sided slopes must both be recovered, if at
all, from estimates uniform in $\eps$.

These two missing structures explain why the fully nonlinear two-phase
problem has remained outside the reach of the existing singular-perturbation
theory. The sharp-interface theory cannot simply be applied before the
compactness needed to produce a sharp interface has been established. In
particular, the scale-uniform Lipschitz
regularity of general sign-changing solutions has long been regarded as a central open problem in the theory. The purpose of
this paper is to resolve this obstruction.

\subsection{Main result and its significance}

\leadefinition{label=SingularPerturbationProblemData, uses={StructuralDataAndNotation,Uniformellipticity,ViscositySolutionConvention}, context={Use n at least 2, 0 < lambda <= Lambda, epsilon > 0, alpha >= 0, beta continuous, nonnegative, nonzero, and compactly supported in (-1,1). The operator acts on real symmetric matrices, is continuous, translation invariant, uniformly elliptic, and normalized by F(0)=0.}}{We consider viscosity solutions of
\begin{equation}\label{eq:singular-perturbation-intro}
F(D^2u_\eps)
=
\frac{\alpha}{\eps}
\beta\left(\frac{u_\eps}{\eps}\right)
\qquad\text{in }B_1\subset\mathbb R^n,
\end{equation}
where $n\geq2$, $\eps>0$, $\alpha\geq0$,
$$
\beta\in C_c (-1,1),
\qquad
\beta\geq0,
\qquad
\beta\not\equiv0,
$$
and $F:\Sn\to\mathbb R$ is continuous, translation invariant, uniformly
elliptic, and normalized by $F(0)=0$. Since $\beta$ is supported in $(-1,1)$,
the equation is homogeneous away from the transition region
$$
\TT_\eps(u_\eps)
:=
\{x\in B_1:|u_\eps(x)|<\eps\}.
$$
}
The $L^\infty$ size of the source is of order $\alpha/\eps$; standard
estimates for equations with bounded right-hand side therefore degenerate as
$\eps\downarrow0$. The following theorem shows that the gradients do not.

\begin{theorem}[Uniform Lipschitz estimate]\label{thm:main-intro}
Let $F:\Sn\to\mathbb R$ be continuous and $(\lambda,\Lambda)$-elliptic, with
$F(0)=0$, and let $u_\eps\in C(B_1)$ be a viscosity solution of
\eqref{eq:singular-perturbation-intro}. Then
\begin{equation}\label{eq:main-estimate-intro}
\|\nabla u_\eps\|_{L^\infty(B_{1/2})}
\leq
C\left(
\|u_\eps\|_{L^\infty(B_1)}+\sqrt\alpha
\right),
\end{equation}
where $C$ depends only on $n$, $\lambda$, $\Lambda$, and $\beta$, and is
independent of $\eps$, $\alpha$, $F$, and $u_\eps$.
\end{theorem}

The estimate is sharp in both scale and form. Across a planar reaction layer,
the change in squared slope is comparable to
$$
\alpha\int_{-1}^{1}\beta(s)\,ds,
$$
which forces a contribution of order $\sqrt\alpha$. In
\Cref{sec:sharpness} we construct radial solutions for the Laplacian such that
$$
\frac{\|u_\eps\|_{L^\infty(B_1)}}{\sqrt\alpha}
\longrightarrow0,
\qquad
\|\nabla u_\eps\|_{L^\infty(B_{1/2})}
\geq c\sqrt\alpha.
$$
Thus the additive term in \eqref{eq:main-estimate-intro} cannot be replaced by
$o(\sqrt\alpha)$.

The content of \Cref{thm:main-intro} is not a regularity estimate for a large
but otherwise harmless source. The equation is genuinely two-phase, the
operator need not possess a variational structure, and the estimate is
obtained directly at the diffuse level without assuming a limiting interface
or a law across it. It therefore provides the compactness input that must
precede, rather than follow from, a sharp-interface analysis. This is the
principal distinction from the existing one-phase theory and from
sharp-interface two-phase regularity results.

Indeed, let $\{u_\eps\}$ be locally uniformly bounded and assume that the
strengths $\alpha$ remain uniformly bounded. Then
\Cref{thm:main-intro}, after localization, makes the family locally
equi-Lipschitz. Along a subsequence,
$$
u_\eps\longrightarrow u
\qquad\text{locally uniformly},
$$
where $u$ is locally Lipschitz and
$$
F(D^2u)=0
\qquad\text{in }
\operatorname{int}\{u>0\}\cup\operatorname{int}\{u<0\}.
$$
Thus the reaction is asymptotically confined to the emerging interface, and
all blow-up sequences lie in a scale-invariant compactness class.

The estimate removes the first and most basic obstruction to the fully
nonlinear two-phase free-boundary program. Its natural continuation is to
establish nondegeneracy for distinguished classes of solutions, determine the
transmission law selected by the approximation, classify global blow-ups, and
develop the regularity theory of the limiting interface. The one-phase
analysis in \cite{RicarteTeixeira} provides a prototype, but the two-phase
problem must also account for the interaction of the limiting slopes and for
the possible large-matrix behavior of $F$. For homogeneous concave or convex operators, the global classification
theorem of De Silva and Savin \cite{DeSilvaSavinGlobal} provides an
important sharp-interface benchmark. The remaining issue in the present
setting is to determine which global profiles are selected by the diffuse
approximation under the general structural assumptions considered here. We do not address these questions
here. The result proved in this paper supplies the uniform scale control on
which that analysis must be built.

\subsection{Architecture of the proof}

We describe the proof with emphasis on the mechanisms that replace the two
structures unavailable in the diffuse problem. Set
$$
A
:=
\|u_\eps\|_{L^\infty(B_1)}+\sqrt\alpha.
$$
After division by $A$ and the corresponding normalization of the operator,
the equation takes the form
$$
F(D^2v)
=
\frac\gamma\delta
\beta\left(\frac v\delta\right),
\qquad
\|v\|_{L^\infty(B_1)}\leq1,
\qquad
0<\gamma,\delta\leq1.
$$
The natural spatial scale of a transition of height $\delta$ and strength
$\gamma$ is
$$
R_{\gamma,\delta}
=
\frac\delta{\sqrt\gamma}.
$$
At this scale, a concave doubling modulus gives an intrinsic
Ishii--Lions estimate in the active layer. It bounds the local slope by
$\sqrt\gamma$ together with the linear growth from the boundary of a slightly
enlarged transition region. A fixed-scale modulus would not generate enough
curvature to absorb the source $\gamma/\delta$.

We encode the remaining boundary growth in a buffered level-boundary
seminorm. The buffer has two roles: it permits localization of the homogeneous
phase estimates, and it prevents maximizing points in the later blow-up
argument from drifting to the artificial boundary of the comparison domain.
Interior estimates in the two homogeneous phases, combined with the
active-layer estimate, reduce the theorem to a universal bound for this
seminorm.

The next ingredient is compactness across diffuse layers. When the effective
strength tends to zero, the source need not converge in any classical norm if
the layer thickness collapses. To recover the homogeneous limiting equation,
we modify touching test functions by small convex correctors whose curvature
in the normal direction dominates the reaction. This curved-test argument
allows viscosity inequalities to pass through the collapsing layer. A
uniform interior $C^{0,\theta}$ estimate supplies the equicontinuity needed
for the compactness step.

In the nondegenerate regime, exact one-dimensional transitions provide the
local comparison geometry. Their incoming and outgoing slopes are related by
a direction-dependent map
$$
a
=
G_{F,\nu,\gamma,\eta}(b).
$$
For a general $F$, this map may retain information about the behavior of the
operator at large matrices; accordingly, we do not assert that the diffuse
approximation selects a unique transmission law at this stage. What is
uniform is the size of the squared-slope increment and the fact that, at large
slopes, $G$ approaches the identity with quantitative control of its first two
derivatives and of its dependence on $\nu$. These are precisely the estimates
needed to invoke the De Silva--Savin decay-versus-Lipschitz alternative after
passage to a sharp-interface limit. The positive-thickness transfer of the
Lipschitz branch follows by the same compactness argument; the dependence on
$\nu$ is handled by tracking the common modulus of continuity of the slope
maps and the uniform constants in the planar barriers.

This yields a diffuse alternative: at each dyadic scale, either the normalized
oscillation decays or the solution is trapped near an affine profile with
controlled slope. The critical issue is that, once the flatness reaches the
natural $\sqrt\gamma$ floor, the available bounds for successive slope changes
need not be summable. No convergence of the affine slopes is inferred from
these bounds. Instead, the proof separates two exhaustive cases. If the
accumulated slopes remain bounded, the affine trapping on dyadic annuli gives
the Lipschitz estimate directly, regardless of whether the slopes converge.
If they are unbounded, we stop the iteration at the first scale at which the
slope $P_j$ becomes large and divide the solution by $P_j$. The effective
strength is then $\gamma_j/P_j^2\to0$. The large-slope branch of the dichotomy
propagates the normalized linear growth back to the macroscopic scale, where
it contradicts the unit-amplitude normalization. Thus the unbounded-slope
case cannot occur.

The resulting anchored dyadic continuation gives uniform linear growth from
every buffered transition point. This bounds the level-boundary seminorm; the
active-layer and homogeneous-phase estimates then yield
\eqref{eq:main-estimate-intro}. Apart from the classical interior regularity
theory for homogeneous uniformly elliptic equations and the quantitative
sharp-interface alternative of De Silva and Savin, the diffuse part of the
argument is self-contained.

\subsection{Acknowledgments, formalization, and machine assistance}
\label{subsec:ack-formalization-ai}

A.S. acknowledges support from King Abdullah University of Science and
Technology (KAUST) under Award No.~ORFS-CRG12-2024-6430. E.V.T. gratefully
acknowledges support from the Grayce B. Kerr Chair funds at Oklahoma State
University.

This research was conducted in part under the DARPA expMath project
\emph{``A Human-Centered Framework for AI-Mathematician Collaboration in
Research-Level Mathematics''} (Agreement No.\ HR0011262E029), in which E.V.T.
serves as a co-principal investigator.
The views, opinions, and/or findings expressed are those of the authors and
should not be interpreted as representing the official views or policies of
the Department of Defense or the U.S.\ Government.
We gratefully acknowledge DARPA's support and the collaborative environment
fostered by the project. We are particularly grateful to our expMath
collaborators at New York University---Daniel Arturi, Jaume de Dios, Chinmay
Hegde, Claudio Silva, and Samuel Westrick---for discussions and contributions
related to formalization, computational verification, and the development of
supporting tools.

Particularly important to our work is the \textsc{Lea} Prover, a tool under
development by our expMath team. \textsc{Lea} is designed as part of a
human-centered approach to AI-assisted mathematics: it streamlines interaction
with frontier models, including OpenAI's GPT, Anthropic's Claude, and Google's
Gemini, and serves as an auxiliary tool to compare alternative formulations,
examine logical dependencies, stress-test candidate arguments, and assist with
proofs.

At a pivotal stage of the project, we employed \textsc{Lea} to explore new approaches to the remaining difficulties. As is common with AI at its current stage of development, most of the ideas produced
were unworkable or incomplete---often locally correct, yet not truly useful. One key
concept survived, however, and turned out to be instrumental in completing the
proof: the concept of the buffered level-boundary seminorm introduced in Subsection \ref{subsec:buffered-level-boundary}. This concept emerged from those explorations and was subsequently refined, developed, and
incorporated into the argument by the authors.

Furthermore, as part of the expMath effort, the core argument underlying the
main theorem was fully formalized in Lean using \textsc{Lea}; the repository is available at \url{https://github.com/teixeiramath-maker/LipReg-TwoPhase-FullyNonlinear}. The
formalization served as a proof-auditing instrument: it made the logical
dependencies explicit and allowed the principal components of the iteration to
be checked systematically. The present manuscript provides the complete
analytic proof, with the additional detail appropriate to the viscosity
framework, and is entirely self-contained.

OpenAI's GPT was also used as auxiliary tools for
routine \LaTeX{} typesetting. All outputs were treated as provisional and
independently verified by the authors. Beyond these uses, the manuscript was
written by the authors, who made all final mathematical and expository
decisions and take full responsibility for its contents.

\subsection{Organization of the paper}

In \Cref{sec:pre} we fix the ellipticity and viscosity conventions, record the
scaling laws and interior estimates, and introduce the level-set notation.
\Cref{sec:normalization-local} reduces the theorem to its normalized form,
defines the buffered level-boundary seminorm, and proves the local estimates
in the active layer and the homogeneous phases.

In \Cref{sec:diffuse-compactness} we establish compactness across diffuse
reaction layers, construct the exact planar transitions and their slope maps,
and prove sharp-interface stability. \Cref{sec:diffuse-continuation} develops
the diffuse De Silva--Savin alternative, including the uniform H\"older
compactness estimate, the finite-thickness transfer, affine continuation, and
the first-passage large-slope argument. In \Cref{sec:completion} we apply the
resulting linear-growth estimate to the buffered level-boundary seminorm and
complete the proof of \Cref{thm:main-intro}. Finally,
\Cref{sec:sharpness} proves the optimality of the $\sqrt\alpha$ term.

\section{Preliminaries}\label{sec:pre}

\leadefinition{label=StructuralDataAndNotation, context={Use n as a natural number with 2 <= n and ellipticity parameters satisfying 0 < lambda <= Lambda. Interpret Sn as the finite-dimensional real vector space of symmetric n by n matrices. Universal constants are existential positive reals whose allowed parameter dependence is exactly the dependence stated in the paper.}}{We collect here the notation and the standard facts used throughout the
paper. We write $B_r(x_0)$ for the Euclidean ball of radius $r$ centered at
$x_0$, and $B_r:=B_r(0)$. The space of real symmetric $n\times n$ matrices is
denoted by $\Sn$. Unless otherwise indicated, constants denoted by $C$ depend
only on $n$, the ellipticity constants $\lambda$ and $\Lambda$, and the fixed
reaction profile $\beta$. We refer to such constants as universal.
}

\begin{definition}[Uniform ellipticity]\label{def:ellipticity}
A continuous operator $F:\Sn\to\mathbb R$ is
$(\lambda,\Lambda)$-elliptic if
\begin{equation}\label{eq:ellipticity}
\lambda\tr N
\leq
F(M+N)-F(M)
\leq
\Lambda\tr N
\end{equation}
for every $M,N\in\Sn$ with $N\geq0$.
\end{definition}

\begin{leatargetblock}
\leadefinition{label=ViscositySolutionConvention, uses={Uniformellipticity}, context={Use one shared viscosity subsolution and supersolution framework throughout the project. With this paper's sign convention, contact from above gives F(D^2 phi) >= f and contact from below gives F(D^2 phi) <= f; all contacts and boundaries are relative to the stated ambient domain. Treat the general viscosity framework as imported foundational analysis.}}
We adopt the convention in \eqref{eq:ellipticity} throughout. Accordingly,
if $\varphi\in C^2$ touches a viscosity solution of
$$
F(D^2u)=f
$$
from above at $x_0$, then
$$
F(D^2\varphi(x_0))\geq f(x_0),
$$
whereas contact from below gives the reverse inequality. All boundaries and
all viscosity inequalities below are understood relative to the ambient
domain unless explicitly stated otherwise; see \cite{CrandallIshiiLions}.

\end{leatargetblock}

\leadefinition{label=PucciExtremalOperators, uses={StructuralDataAndNotation}}{If $e_1,\ldots,e_n$ are the eigenvalues of $M\in\Sn$, the Pucci extremal
operators corresponding to this convention are
$$
\mathcal P^+_{\lambda,\Lambda}(M)
:=
\Lambda\sum_{e_i>0}e_i
+
\lambda\sum_{e_i<0}e_i
$$
and
$$
\mathcal P^-_{\lambda,\Lambda}(M)
:=
\lambda\sum_{e_i>0}e_i
+
\Lambda\sum_{e_i<0}e_i.
$$
}%
\lealemma{label=UniformEllipticityPucciBounds, uses={Uniformellipticity,PucciExtremalOperators}}{Uniform ellipticity is equivalently expressed by
\begin{equation}\label{eq:pucci-difference}
\mathcal P^-_{\lambda,\Lambda}(M-N)
\leq
F(M)-F(N)
\leq
\mathcal P^+_{\lambda,\Lambda}(M-N)
\end{equation}
for all $M,N\in\Sn$. Since all operators considered in the paper are
normalized by $F(0)=0$, it follows in particular that
\begin{equation}\label{eq:pucci-control}
\mathcal P^-_{\lambda,\Lambda}(M)
\leq
F(M)
\leq
\mathcal P^+_{\lambda,\Lambda}(M).
\end{equation}
}%
\begin{leatargetcontinuation}\leatheorem{label=EllipticOperatorCompactnessAndViscosityStability, uses={UniformEllipticityPucciBounds,ViscositySolutionConvention}, context={Treat local compactness of normalized uniformly elliptic operators and stability of viscosity solutions under locally uniform convergence as imported foundational results. Use the topology of locally uniform convergence on the finite-dimensional space of symmetric matrices.}}The inequalities in \eqref{eq:pucci-difference} also imply that normalized
$(\lambda,\Lambda)$-elliptic operators form a locally compact family for the
topology of locally uniform convergence on $\Sn$. We shall use this elementary
compactness, together with the standard stability of viscosity solutions,
without further comment.

\end{leatargetcontinuation}

\leadefinition{label=SingularPerturbationRescaling, uses={StructuralDataAndNotation}}{We record explicitly the scaling convention. Given $x_0\in\mathbb R^n$,
$r,A>0$, and $c\in\mathbb R$, set
\begin{equation}\label{eq:general-rescaling}
v(x):=\frac{u(x_0+rx)-c}{A},
\qquad
F_{r,A}(M):=\frac{r^2}{A}
F\left(\frac{A}{r^2}M\right).
\end{equation}
}%
\lealemma{label=lemSingularPerturbationRescaling, uses={SingularPerturbationRescaling,Uniformellipticity,SingularPerturbationProblemData}}{Then $F_{r,A}(0)=0$ and $F_{r,A}$ has the same ellipticity constants as $F$.
If
$$
F(D^2u)
=
\frac{\alpha}{\eps}
\beta\left(\frac{u}{\eps}\right),
$$
then $v$ satisfies
\begin{equation}\label{eq:rescaled-singular-equation}
F_{r,A}(D^2v)
=
\frac{\gamma}{\delta}
\beta\left(\frac{v+\sigma}{\delta}\right),
\qquad
\gamma:=\frac{r^2\alpha}{A^2},
\quad
\delta:=\frac{\eps}{A},
\quad
\sigma:=\frac{c}{A}.
\end{equation}
Thus spatial and amplitude rescalings preserve the class of equations, while
the effective strength transforms as $\gamma=r^2\alpha/A^2$. This identity
will be used repeatedly in the blow-up arguments.
}

\leatheorem{label=InteriorC1AlphaEstimate, uses={Uniformellipticity,ViscositySolutionConvention,PointwiseUpperLipschitz}, context={This is the classical interior C^{1,alpha} estimate for uniformly elliptic viscosity equations from Caffarelli--Cabre. Treat it, including its homogeneous scale-invariant affine-approximation form, as an imported analytic theorem rather than reproving it in this project.}}{The classical regularity input used throughout the proof is the interior
$C^{1,\bar\alpha}$ estimate for uniformly elliptic equations; see
\cite{CaffarelliCabre}. There exist $\bar\alpha\in(0,1)$ and $C<\infty$,
depending only on $n$, $\lambda$, and $\Lambda$, such that every viscosity
solution of
$$
F(D^2h)=f
\qquad\text{in }B_1
$$
satisfies
\begin{equation}\label{eq:interior-gradient-estimate}
\norm{\nabla h}_{L^\infty(B_{1/2})}
+
[\nabla h]_{C^{0,\bar\alpha}(B_{1/2})}
\leq
C\left(
\norm{h}_{L^\infty(B_1)}
+
\norm{f}_{L^\infty(B_1)}
\right).
\end{equation}
In the homogeneous case, the scale-invariant form of the same estimate gives
\begin{equation}\label{eq:homogeneous-affine-approximation}
\norm{
h-h(x_0)-\nabla h(x_0)\cdot(x-x_0)
}_{L^\infty(B_\rho(x_0))}
\leq
C\rho^{1+\bar\alpha}\osc_{B_1}h
\end{equation}
for $x_0\in B_{1/2}$ and $0<\rho\leq1/4$, whenever
$F(D^2h)=0$ in $B_1$.
}

\leadefinition{label=TransitionRegion, uses={StructuralDataAndNotation}}{Finally, let $v\in C(\Omega)$ and $a>0$. We denote its transition region at
height $a$ and the corresponding level boundary by
\begin{equation}\label{eq:layer-def}
\TT_a(v;\Omega)
:=
\{x\in\Omega:|v(x)|<a\},
\qquad
\GG_a(v;\Omega)
:=
\partial_\Omega\TT_a(v;\Omega).
\end{equation}
When the ambient domain is clear, it will be suppressed from the notation.
}%
\leadefinition{label=BoundaryLipschitzSeminorm, uses={StructuralDataAndNotation}}{For $E\subset\Omega$, we define the boundary-to-domain Lipschitz seminorm
\begin{equation}\label{eq:bdry-lip-def}
[v]_{\Lip(E;\Omega)}
:=
\sup_{z\in E}
\sup_{x\in\Omega\setminus\{z\}}
\frac{|v(x)-v(z)|}{|x-z|}.
\end{equation}
Fixing one endpoint on $E$ makes this quantity particularly suited to
measuring linear growth away from a transition boundary.
}%
\leadefinition{label=PointwiseUpperLipschitz, uses={StructuralDataAndNotation}, context={For every theorem whose hypotheses initially give only continuity, formalize a displayed L-infinity gradient conclusion as the corresponding Lipschitz or pointwise upper-Lipschitz bound. Once Lipschitz regularity is established, the almost-everywhere gradient formulation follows separately.}}{We shall also use the pointwise upper Lipschitz constant
\begin{equation}\label{eq:point-lip-def}
\Lip_\Omega v(x_0)
:=
\limsup_{\substack{x\to x_0\\ x\in\Omega\setminus\{x_0\}}}
\frac{|v(x)-v(x_0)|}{|x-x_0|}.
\end{equation}
At every point of differentiability of $v$, one has
$\Lip_\Omega v(x_0)=|\nabla v(x_0)|$.
}

\section{Normalization and local estimates}
\label{sec:normalization-local}

We first reduce the main theorem to a unit-amplitude estimate. We then control
the gradient separately in the transition region and in the two homogeneous
phases. These local estimates show that the proof reduces to linear growth
from the boundary of a fixed enlargement of the reaction layer.

\subsection{Normalization}

The normalized statement is as follows.

\begin{theorem}[Normalized Lipschitz estimate]\label{thm:normalized}
Let $F\colon \Sn\to\mathbb R$ be continuous and
$(\lambda,\Lambda)$-elliptic, with $F(0)=0$. Suppose that
$0<\gamma\leq1$, $0<\delta\leq1$, and that $v\in C(B_1)$ is a viscosity
solution of
\leadefinition{label=NormalizedDiffuseProblem, uses={StructuralDataAndNotation,Uniformellipticity,ViscositySolutionConvention,SingularPerturbationProblemData}, context={This target packages the normalized diffuse equation together with the surrounding assumptions in NormalizedEstimate: F is continuous, (lambda,Lambda)-elliptic, F(0)=0, 0 < gamma <= 1, 0 < delta <= 1, v is a continuous viscosity solution on B_1, and ||v||_{L-infinity(B_1)} <= 1.}}{\begin{equation}\label{eq:normalized}
F(D^2v)
=
\frac{\gamma}{\delta}
\beta\left(\frac{v}{\delta}\right)
\qquad\text{in }B_1,
\qquad
\norm{v}_{L^\infty(B_1)}\leq1.
\end{equation}
}
Then
\begin{equation}\label{eq:normalized-bound}
\norm{\nabla v}_{L^\infty(B_{1/2})}
\leq C.
\end{equation}
\end{theorem}

\begin{proposition}[Reduction to the normalized problem]
\label{prop:normalization}
\Cref{thm:normalized} implies \Cref{thm:main-intro}.
\end{proposition}
\begin{proof}
Let
$$
A
:=
\norm{u_\eps}_{L^\infty(B_1)}+\sqrt\alpha.
$$
If $A=0$, then $u_\eps\equiv0$ and there is nothing to prove. Assume
$A>0$ and set
$$
v:=\frac{u_\eps}{A},
\qquad
G(M):=\frac1A F(AM),
\qquad
\delta:=\frac\eps A,
\qquad
\gamma:=\frac\alpha{A^2}.
$$
This is the rescaling \eqref{eq:general-rescaling} with $r=1$ and $c=0$.
Consequently, $G(0)=0$, the operator $G$ has the same ellipticity constants as
$F$, and
\begin{equation}\label{eq:normalized-rescaling}
G(D^2v)
=
\frac{\gamma}{\delta}
\beta\left(\frac{v}{\delta}\right)
\qquad\text{in }B_1.
\end{equation}
Moreover,
$$
\norm{v}_{L^\infty(B_1)}\leq1,
\qquad
0\leq\gamma\leq1.
$$

If $\alpha=0$, equation \eqref{eq:normalized-rescaling} is homogeneous and
the conclusion follows directly from \eqref{eq:interior-gradient-estimate}.
We may therefore assume that $\gamma>0$. If $\delta\leq1$, we apply
\Cref{thm:normalized} to $v$. If $\delta>1$, then
$$
\left\|
\frac{\gamma}{\delta}
\beta\left(\frac{v}{\delta}\right)
\right\|_{L^\infty(B_1)}
\leq
\norm{\beta}_{L^\infty(\mathbb R)},
$$
and \eqref{eq:interior-gradient-estimate} again gives
$\norm{\nabla v}_{L^\infty(B_{1/2})}\leq C$. Since $u_\eps=Av$, both cases
yield
$$
\norm{\nabla u_\eps}_{L^\infty(B_{1/2})}
\leq
C\left(
\norm{u_\eps}_{L^\infty(B_1)}+\sqrt\alpha
\right),
$$
which is \eqref{eq:main-estimate-intro}.
\end{proof}

\subsection{The buffered level-boundary seminorm}\label{subsec:buffered-level-boundary}

\leadefinition{label=BufferedBoundarySeminorm, uses={NormalizedDiffuseProblem,TransitionRegion,BoundaryLipschitzSeminorm}}{For definiteness, fix $c_0:=2$. Since
$\operatorname{supp}\beta\Subset(-1,1)$, the equation in
\eqref{eq:normalized} is homogeneous on $\{|v|\geq\delta\}$; the larger set
$\{|v|<c_0\delta\}$ therefore provides a fixed buffer around the active
reaction layer. In the notation of \eqref{eq:layer-def}, set
$$
\Gamma_\delta(v)
:=
\GG_{c_0\delta}(v;B_1)
=
\partial_{B_1}\{|v|<c_0\delta\}.
$$
We shall measure growth from this level boundary by
\begin{equation}\label{eq:Bint}
\mathcal B(v)
:=
\sup_{z\in\Gamma_\delta(v)\cap B_{2/3}}
\sup_{x\in B_{3/4}\setminus\{z\}}
\left(\frac23-|z|\right)
\frac{|v(x)-v(z)|}{|x-z|},
\end{equation}
with the convention that the supremum is zero when
$\Gamma_\delta(v)\cap B_{2/3}=\varnothing$.
}

The weight $2/3-|z|$ has two complementary roles. On $B_{7/12}$ it is
bounded below by $1/12$, so \eqref{eq:Bint} gives an unweighted linear-growth
estimate near every level-boundary point relevant to $B_{1/2}$. In the later
blow-up argument, the same weight forces the selected scale to be negligible
relative to the distance of the base point from $\partial B_{2/3}$, thereby
preventing the rescaled domains from seeing this artificial boundary.

\lealemma{label=ScaledInteriorGradientEstimate, uses={InteriorC1AlphaEstimate,PointwiseUpperLipschitz}, context={Derive this by translating and scaling the imported unit-ball interior C^{1,alpha} estimate.}}{We shall repeatedly use the following scaled consequence of
\eqref{eq:interior-gradient-estimate}: if
$$
F(D^2w)=f
\qquad\text{in }B_r(x_0)\Subset B_{3/4},
$$
then
\begin{equation}\label{eq:scaled-interior-gradient}
\Lip_{B_{3/4}}w(x_0)
\leq
C\left(
\frac{\osc_{B_r(x_0)}w}{r}
+
r\norm{f}_{L^\infty(B_r(x_0))}
\right).
\end{equation}
}

\subsection{The transition region}

\leadefinition{label=IntrinsicTransitionRadius, uses={StructuralDataAndNotation}}{The intrinsic radius of a transition of height $\delta$ and strength
$\gamma$ is
$$
R_{\gamma,\delta}
:=
\frac{\delta}{\sqrt\gamma}.
$$
Indeed, an oscillation of size $\delta$ over a ball of radius $R$ produces a
second-order contribution of size $\delta/R^2$, which balances the forcing
$\gamma/\delta$ precisely when $R=R_{\gamma,\delta}$. The corresponding
first-order scale is $\delta/R_{\gamma,\delta}=\sqrt\gamma$.
}

\begin{lemma}[Transition-region estimate]\label{lem:active}
Let $v$ satisfy \eqref{eq:normalized}. Then
\begin{equation}\label{eq:active}
\sup_{x\in B_{1/2}\cap\{|v|\leq c_0\delta\}}
\Lip_{B_{3/4}}v(x)
\leq
C\left(1+\mathcal B(v)\right).
\end{equation}
More precisely, there exists a universal $r_*>0$ such that, whenever
$R_{\gamma,\delta}\leq r_*$,
\begin{equation}\label{eq:active-microscopic}
\sup_{x\in B_{1/2}\cap\{|v|\leq c_0\delta\}}
\Lip_{B_{3/4}}v(x)
\leq
C\left(\sqrt\gamma+\mathcal B(v)\right).
\end{equation}
\end{lemma}

\begin{proof}
Set $R:=R_{\gamma,\delta}$ and fix
$x_0\in B_{1/2}\cap\{|v|\leq c_0\delta\}$. We take $r_*:=1/8$.

Suppose first that $R\geq r_*$. Since $\sqrt\gamma\leq1$,
$$
\frac{\gamma}{\delta}
=
\frac{\sqrt\gamma}{R}
\leq
\frac1{r_*}.
$$
The right-hand side of \eqref{eq:normalized} is therefore universally
bounded. Applying \eqref{eq:scaled-interior-gradient} on a fixed ball centered
at $x_0$ gives
$$
\Lip_{B_{3/4}}v(x_0)\leq C.
$$

Assume now that $R<r_*$. If $x_0\in\Gamma_\delta(v)$, then
\eqref{eq:Bint} and $2/3-|x_0|\geq1/6$ give directly
$$
\Lip_{B_{3/4}}v(x_0)
\leq
C\mathcal B(v).
$$
We may thus suppose that $|v(x_0)|<c_0\delta$, and set
$$
d:=\dist\bigl(x_0,\Gamma_\delta(v)\bigr),
$$
with $d=+\infty$ if the level boundary is empty.

If $d\leq R$, choose $z\in\Gamma_\delta(v)$ with $|x_0-z|=d$. Since
$x_0\in B_{1/2}$ and $R<1/8$, we have $z\in B_{5/8}\subset B_{2/3}$ and
$$
\frac23-|z|\geq\frac1{24}.
$$
For every $y\in B_{d/2}(x_0)$, \eqref{eq:Bint} yields
$$
|v(y)-v(z)|
\leq
C\mathcal B(v)|y-z|
\leq
C\mathcal B(v)d.
$$
Consequently,
$$
\osc_{B_{d/2}(x_0)}v
\leq
C\mathcal B(v)d.
$$
Using \eqref{eq:scaled-interior-gradient} and $d\leq R$, we obtain
$$
\Lip_{B_{3/4}}v(x_0)
\leq
C\mathcal B(v)
+
Cd\frac{\gamma}{\delta}
\leq
C\left(\mathcal B(v)+\sqrt\gamma\right).
$$

If $d>R$, continuity implies that $B_R(x_0)$ lies in the same component of
$\{|v|<c_0\delta\}$ as $x_0$. Hence
$$
\osc_{B_R(x_0)}v
\leq
2c_0\delta.
$$
Applying \eqref{eq:scaled-interior-gradient} on a concentric ball of radius
comparable to $R$ gives
$$
\Lip_{B_{3/4}}v(x_0)
\leq
C\left(
\frac{\delta}{R}
+
R\frac{\gamma}{\delta}
\right)
=
C\sqrt\gamma.
$$
This proves \eqref{eq:active-microscopic}; together with the case
$R\geq r_*$, it also proves \eqref{eq:active}.
\end{proof}

\subsection{The homogeneous phases}

Outside the enlarged transition region the reaction vanishes. The gradient
there is controlled by the oscillation on a ball whose radius is comparable
to the distance from $\Gamma_\delta(v)$; the seminorm $\mathcal B(v)$ provides
exactly the required oscillation bound.

\begin{lemma}[Homogeneous-phase estimate]\label{lem:phases}
Let $v$ satisfy \eqref{eq:normalized}. Then
\begin{equation}\label{eq:phases}
\sup_{x\in B_{1/2}\cap\{|v|>c_0\delta\}}
\Lip_{B_{3/4}}v(x)
\leq
C\left(1+\mathcal B(v)\right).
\end{equation}
\end{lemma}

\begin{proof}
Fix $x_0\in B_{1/2}\cap\{|v|>c_0\delta\}$, and let $D$ be the connected
component of
$$
B_{3/4}\cap\{|v|>c_0\delta\}
$$
containing $x_0$. Since $c_0>1$ and
$\operatorname{supp}\beta\subset(-1,1)$,
$$
F(D^2v)=0
\qquad\text{in }D.
$$
Set $d:=\dist(x_0,\partial D)$.

If $d\geq1/12$, then $B_{1/12}(x_0)\subset D$, and
\eqref{eq:scaled-interior-gradient}, together with
$\norm{v}_{L^\infty(B_1)}\leq1$, gives
$$
\Lip_{B_{3/4}}v(x_0)\leq C.
$$

Suppose that $d<1/12$, and choose $z\in\partial D$ with $|x_0-z|=d$.
Because $x_0\in B_{1/2}$,
$$
|z|<\frac12+\frac1{12}=\frac7{12}.
$$
Thus $z\notin\partial B_{3/4}$ and
$z\in\Gamma_\delta(v)\cap B_{2/3}$. For every
$y\in B_{d/2}(x_0)$, \eqref{eq:Bint} gives
$$
|v(y)-v(z)|
\leq
C\mathcal B(v)d,
$$
and therefore
$$
\osc_{B_{d/2}(x_0)}v
\leq
C\mathcal B(v)d.
$$
The homogeneous form of \eqref{eq:scaled-interior-gradient} now yields
$$
\Lip_{B_{3/4}}v(x_0)
\leq
C\mathcal B(v).
$$
The two cases prove \eqref{eq:phases}.
\end{proof}

The preceding estimates reduce the normalized theorem to a uniform bound for
$\mathcal B(v)$.

\begin{corollary}[Reduction to transition-boundary growth]
\label{cor:reduction}
Every solution of \eqref{eq:normalized} satisfies
\begin{equation}\label{eq:reduction}
\norm{\nabla v}_{L^\infty(B_{1/2})}
\leq
C\left(1+\mathcal B(v)\right).
\end{equation}
\end{corollary}

\begin{proof}
The sets $\{|v|\leq c_0\delta\}$ and
$\{|v|>c_0\delta\}$ partition $B_{1/2}$. Apply
Lemmas \ref{lem:active} and \ref{lem:phases} on the respective regions.
\end{proof}

\section{Diffuse compactness and planar transitions}
\label{sec:diffuse-compactness}

\leadefinition{label=ShiftedDiffuseProblem, uses={SingularPerturbationProblemData,lemSingularPerturbationRescaling,BufferedBoundarySeminorm}, context={This is the shared shifted diffuse equation used in subsequent compactness and continuation targets. The operator is normalized and uniformly elliptic, gamma is the effective strength, eta > 0 is the thickness, and |sigma| <= c_0 eta.}}{The blow-ups used later preserve the form of the equation but generally
introduce a vertical shift. We are therefore led to consider
\begin{equation}\label{eq:shifted}
F(D^2w)
=
\frac{\gamma}{\eta}
\beta\left(\frac{w+\sigma}{\eta}\right),
\qquad
|\sigma|\leq c_0\eta.
\end{equation}
Here $\eta$ is the rescaled thickness, $\gamma$ is the effective strength,
and $\sigma$ records the value subtracted when the blow-up is centered. The
bound on $\sigma$ is precisely what results from centering at a point of the
enlarged transition region introduced in the preceding section.
}

This section develops two complementary compactness principles. If the
effective strength tends to zero, the reaction disappears even when its
$L^\infty$ norm does not. If instead the thickness collapses while the
strength remains bounded, exact planar profiles carry the slope balance of the
diffuse layer into the limiting two-phase problem.

\subsection{Compactness at vanishing strength}

The first result is the curved-test compactness lemma used in the
large-slope normalization. Its content is that the smallness of $\gamma_j$,
rather than that of $\gamma_j/\eta_j$, is the relevant condition.

\begin{lemma}[Vanishing-strength compactness]\label{lem:compactness}
Let $\Omega\subset\mathbb R^n$ be open. Suppose that $F_j:\Sn\to\mathbb R$
are $(\lambda,\Lambda)$-elliptic, $F_j(0)=0$, and
$F_j\to F_\infty$ locally uniformly on $\Sn$. Let $w_j\in C(\Omega)$ solve
\begin{equation}\label{eq:vanishing-strength-sequence}
F_j(D^2w_j)
=
\frac{\gamma_j}{\eta_j}
\beta\left(\frac{w_j+\sigma_j}{\eta_j}\right)
\qquad\text{in }\Omega,
\end{equation}
where
$$
\gamma_j\longrightarrow0,
\qquad
\eta_j>0,
\qquad
|\sigma_j|\leq c_0\eta_j.
$$
If $w_j\to w$ locally uniformly in $\Omega$, then
\begin{equation}\label{eq:homogeneous-limit}
F_\infty(D^2w)=0
\qquad\text{in }\Omega.
\end{equation}
\end{lemma}

\begin{proof}
Let $\phi\in C^2$ touch $w$ strictly from above at $x_0\in\Omega$. By the
usual stability argument, there exist $x_j\to x_0$ at which a small vertical
translate of $\phi$ touches $w_j$ from above. Since the right-hand side of
\eqref{eq:vanishing-strength-sequence} is nonnegative,
$$
F_j(D^2\phi(x_j))\geq0.
$$
Passing to the limit gives
$F_\infty(D^2\phi(x_0))\geq0$.

It remains to prove the inequality for lower tests. Let $\phi$ touch $w$
strictly from below at $x_0$, and suppose, toward a contradiction, that
$$
F_\infty(D^2\phi(x_0))>0.
$$
Choose a ball $B_r(x_0)\Subset\Omega$ on which the contact is strict. By a
standard small tilt, followed by a vertical translation and a limiting
argument, we may assume that
\begin{equation}\label{eq:nonzero-test-gradient}
|\nabla\phi|\geq m>0
\qquad\text{in }B_r(x_0).
\end{equation}
Shrinking $r$ and taking $j$ large, we may also arrange that
\begin{equation}\label{eq:strict-test-operator}
F_j(D^2\phi)\geq4a>0
\qquad\text{in }B_r(x_0)
\end{equation}
for some $a>0$ independent of $j$.

If $\liminf_j\eta_j>0$, then the right-hand sides in
\eqref{eq:vanishing-strength-sequence} converge uniformly to zero, and the
ordinary stability theorem already gives a contradiction. We may therefore
pass to a subsequence for which $\eta_j\to0$.

Choose a smooth convex function $h:\mathbb R\to\mathbb R$ such that
$$
h(0)=0,
\qquad
\norm{h'}_{L^\infty(\mathbb R)}<\infty,
\qquad
\lambda m^2h''(s)
\geq
4\norm{\beta}_{L^\infty(\mathbb R)}
\quad\text{for }|s|\leq2.
$$
For $t\in\mathbb R$, define
$$
s_{j,t}(x)
:=
\frac{\phi(x)+t+\sigma_j}{\eta_j},
\qquad
\Phi_{j,t}(x)
:=
\phi(x)+t+\gamma_j\eta_jh(s_{j,t}(x)).
$$
Because $h'$ is bounded and $h(0)=0$, the correction tends to zero locally
uniformly. Moreover,
$$
\partial_t\Phi_{j,t}
=
1+\gamma_jh'(s_{j,t})
\geq\frac12
$$
for all large $j$. We may therefore slide the family $\Phi_{j,t}$ upward
until it first touches $w_j$. The strict boundary separation inherited from
the contact of $\phi$ with $w$ ensures that the first contact point
$x_j$ lies in $B_r(x_0)$ and satisfies $x_j\to x_0$.

At $x_j$, writing $s_j=s_{j,t_j}(x_j)$, we have
\begin{equation}\label{eq:corrected-test-hessian}
D^2\Phi_{j,t_j}
=
\bigl(1+\gamma_jh'(s_j)\bigr)D^2\phi
+
\frac{\gamma_j}{\eta_j}h''(s_j)
\nabla\phi\otimes\nabla\phi.
\end{equation}
The argument of $\beta$ in the equation at the contact point is
$$
\frac{\Phi_{j,t_j}(x_j)+\sigma_j}{\eta_j}
=
s_j+\gamma_jh(s_j).
$$
If the reaction is nonzero there, then
$|s_j+\gamma_jh(s_j)|<1$. Since $h$ is globally Lipschitz and $h(0)=0$,
this implies $|s_j|\leq2$ for all large $j$.

Using \eqref{eq:pucci-difference},
\eqref{eq:nonzero-test-gradient}, and
\eqref{eq:strict-test-operator}, we obtain
$$
F_j(D^2\Phi_{j,t_j})
\geq
4a-C\gamma_j
+
\lambda\frac{\gamma_j}{\eta_j}
h''(s_j)|\nabla\phi(x_j)|^2.
$$
If the reaction vanishes, the right-hand side is positive for large $j$,
contradicting the viscosity inequality for a lower test. If the reaction is
nonzero, the choice of $h$ gives
$$
F_j(D^2\Phi_{j,t_j})
>
\frac{\gamma_j}{\eta_j}
\norm{\beta}_{L^\infty(\mathbb R)}
\geq
\frac{\gamma_j}{\eta_j}
\beta\left(
\frac{\Phi_{j,t_j}(x_j)+\sigma_j}{\eta_j}
\right),
$$
which is again impossible. Hence every lower test satisfies
$F_\infty(D^2\phi(x_0))\leq0$. Letting the auxiliary tilt tend to zero proves
\eqref{eq:homogeneous-limit}.
\end{proof}

\subsection{Exact planar transitions}
We next construct the one-dimensional profiles that describe the passage
through the diffuse reaction layer. These profiles provide the comparison
geometry used when the thickness collapses. They do not, in general, determine
a unique limiting transmission law: the slope increment may retain information
about the behavior of $F$ at large matrices. What is universal, and sufficient
for the argument, is the size of this increment, its dependence on the normal
direction, and its asymptotic behavior at large slopes.

\leadefinition{label=PlanarRankOneOperator, uses={StructuralDataAndNotation}}{Fix $\nu\in\mathbb S^{n-1}$ and define
$$
Q_{F,\nu}(t)
:=
F(t\nu\otimes\nu),
\qquad t\geq0.
$$
}%
\lealemma{label=lemPlanarOperatorBounds, uses={Uniformellipticity,PlanarRankOneOperator}}{Since $F(0)=0$ and $F$ is uniformly elliptic, $Q_{F,\nu}$ is strictly
increasing from $[0,\infty)$ onto $[0,\infty)$, and
\begin{equation}\label{eq:Qbounds}
\lambda t
\leq
Q_{F,\nu}(t)
\leq
\Lambda t
\qquad\text{for }t\geq0.
\end{equation}
In particular, $Q_{F,\nu}^{-1}$ is well defined on $[0,\infty)$ and satisfies
\begin{equation}\label{eq:Qinverse-bounds}
\frac{y}{\Lambda}
\leq
Q_{F,\nu}^{-1}(y)
\leq
\frac{y}{\lambda}
\qquad\text{for }y\geq0.
\end{equation}
}

\lealemma{label=lemPlanarOperatorDirectionalBounds, uses={UniformEllipticityPucciBounds,PlanarRankOneOperator,lemPlanarOperatorBounds}}{We shall also need quantitative control of the dependence on the direction.
By \eqref{eq:pucci-difference},
\begin{equation}\label{eq:Qdirection}
\left|
Q_{F,\nu}(t)-Q_{F,\widetilde\nu}(t)
\right|
\leq
Ct|\nu-\widetilde\nu|
\qquad\text{for }t\geq0
\end{equation}
and $\nu,\widetilde\nu\in\mathbb S^{n-1}$. Consequently,
\begin{equation}\label{eq:Qinverse-direction}
\left|
Q_{F,\nu}^{-1}(y)
-
Q_{F,\widetilde\nu}^{-1}(y)
\right|
\leq
Cy|\nu-\widetilde\nu|
\qquad\text{for }y\geq0.
\end{equation}
Indeed, let
$$
r:=Q_{F,\nu}^{-1}(y),
\qquad
\widetilde r:=Q_{F,\widetilde\nu}^{-1}(y),
$$
and suppose, without loss of generality, that $r\geq\widetilde r$. Then
\begin{align*}
\lambda(r-\widetilde r)
&\leq
Q_{F,\nu}(r)-Q_{F,\nu}(\widetilde r)
\\
&=
Q_{F,\widetilde\nu}(\widetilde r)
-
Q_{F,\nu}(\widetilde r)
\\
&\leq
C\widetilde r|\nu-\widetilde\nu|
\leq
Cy|\nu-\widetilde\nu|,
\end{align*}
where the last inequality follows from
\eqref{eq:Qinverse-bounds}. This proves
\eqref{eq:Qinverse-direction}.
}

\leadefinition{label=PlanarTransitionIncrement, uses={SingularPerturbationProblemData,PlanarRankOneOperator,lemPlanarOperatorBounds}}{For $\gamma,\eta>0$, define
\begin{equation}\label{eq:J}
J_{F,\nu}(\gamma,\eta)
:=
2\eta
\int_{-1}^{1}
Q_{F,\nu}^{-1}\left(
\frac{\gamma}{\eta}\beta(s)
\right)\,ds.
\end{equation}
}%
\lealemma{label=lemTransitionIncrementBounds, uses={PlanarTransitionIncrement,lemPlanarOperatorBounds,lemPlanarOperatorDirectionalBounds}}{The bounds in \eqref{eq:Qinverse-bounds} give
\begin{equation}\label{eq:Jbounds}
\frac{2\gamma}{\Lambda}
\int_{-1}^{1}\beta(s)\,ds
\leq
J_{F,\nu}(\gamma,\eta)
\leq
\frac{2\gamma}{\lambda}
\int_{-1}^{1}\beta(s)\,ds.
\end{equation}
In particular,
$$
J_{F,\nu}(\gamma,\eta)\asymp\gamma
$$
uniformly in $\eta$ and $\nu$. Moreover,
\eqref{eq:Qinverse-direction} yields
\begin{equation}\label{eq:Jdirection}
\left|
J_{F,\nu}(\gamma,\eta)
-
J_{F,\widetilde\nu}(\gamma,\eta)
\right|
\leq
C\gamma|\nu-\widetilde\nu|.
\end{equation}
}

\leadefinition{label=ExactPlanarProfile, uses={ShiftedDiffuseProblem,PlanarRankOneOperator,lemPlanarOperatorBounds}}{We now construct the planar profile. Let $b>0$ be the incoming slope and set
\begin{equation}\label{eq:planar-q}
q(s)^2
:=
b^2
+
2\eta
\int_{-1}^{s}
Q_{F,\nu}^{-1}\left(
\frac{\gamma}{\eta}\beta(\tau)
\right)\,d\tau,
\qquad -1\leq s\leq1.
\end{equation}
Thus $q(-1)=b$ and $q(s)\geq b$. Define
\begin{equation}\label{eq:planar-time-change}
t(s)
:=
\eta\int_0^s\frac{d\tau}{q(\tau)}.
\end{equation}
The function $t$ is strictly increasing. Let $s=s(t)$ denote its inverse and
define
\begin{equation}\label{eq:planar-profile}
\varphi(t)
:=
\eta s(t)-\sigma
\qquad
\text{for }t\in[t(-1),t(1)].
\end{equation}
}%
\lealemma{label=lemExactPlanarEquation, uses={ExactPlanarProfile,lemPlanarOperatorBounds}}{Since
$$
\frac{dt}{ds}=\frac{\eta}{q(s)},
$$
we have
\begin{equation}\label{eq:planar-profile-derivatives}
\varphi'(t)=q(s),
\qquad
\varphi''(t)
=
Q_{F,\nu}^{-1}\left(
\frac{\gamma}{\eta}\beta(s)
\right),
\qquad
s=\frac{\varphi(t)+\sigma}{\eta}.
\end{equation}
Consequently, the planar function
$$
\Phi(x):=\varphi(x\cdot\nu)
$$
satisfies
\begin{align*}
F(D^2\Phi)
&=
F\left(
\varphi''(x\cdot\nu)\nu\otimes\nu
\right)
\\
&=
Q_{F,\nu}\left(
Q_{F,\nu}^{-1}\left(
\frac{\gamma}{\eta}
\beta\left(\frac{\Phi+\sigma}{\eta}\right)
\right)
\right)
\\
&=
\frac{\gamma}{\eta}
\beta\left(\frac{\Phi+\sigma}{\eta}\right)
\end{align*}
throughout the transition interval.
}

\lealemma{label=lemExactPlanarTransition, uses={ExactPlanarProfile,lemExactPlanarEquation,PlanarTransitionIncrement,lemPlanarOperatorBounds}}{Because $\beta$ vanishes near $\{-1,1\}$, the function $\varphi$ extends
linearly outside $[t(-1),t(1)]$, with slope $b$ on the incoming side and
slope
$$
a:=q(1)
$$
on the outgoing side. The resulting planar function is a global smooth
solution of the shifted diffuse equation. Evaluating
\eqref{eq:planar-q} at $s=1$ gives the exact slope balance
\begin{equation}\label{eq:jump}
a^2-b^2
=
J_{F,\nu}(\gamma,\eta).
\end{equation}
}

\leadefinition{label=PlanarTransmissionMap, uses={PlanarTransitionIncrement}}{Accordingly, the transmission map carried by the planar layer is
\begin{equation}\label{eq:G}
G_{F,\nu,\gamma,\eta}(b)
:=
\sqrt{
b^2+J_{F,\nu}(\gamma,\eta)
}.
\end{equation}
}%
\lealemma{label=lemTransmissionMapBasicEstimates, uses={PlanarTransmissionMap,lemTransitionIncrementBounds}}{For brevity, write
$$
G(b):=G_{F,\nu,\gamma,\eta}(b),
\qquad
J:=J_{F,\nu}(\gamma,\eta).
$$
Then
\begin{equation}\label{eq:Gderivatives}
G'(b)
=
\frac{b}{\sqrt{b^2+J}},
\qquad
G''(b)
=
\frac{J}{(b^2+J)^{3/2}}.
\end{equation}
Using \eqref{eq:Jbounds}, we obtain, for every $b>0$,
\begin{equation}\label{eq:Gidentity}
|G(b)-b|
\leq
C\frac{\gamma}{b},
\qquad
|G'(b)-1|
\leq
C\frac{\gamma}{b^2},
\qquad
|G''(b)|
\leq
C\frac{\gamma}{b^3}.
\end{equation}
Indeed,
$$
G(b)-b
=
\frac{J}{\sqrt{b^2+J}+b},
$$
and the remaining estimates follow directly from
\eqref{eq:Gderivatives}.
}

\lealemma{label=lemTransmissionMapDirectionalEstimate, uses={PlanarTransmissionMap,lemTransitionIncrementBounds}}{The directional modulus of $J$ in \eqref{eq:Jdirection} also gives
\begin{equation}\label{eq:Gdirection}
\left|
G_{F,\nu,\gamma,\eta}(b)
-
G_{F,\widetilde\nu,\gamma,\eta}(b)
\right|
\leq
C\frac{\gamma}{b}
|\nu-\widetilde\nu|
\qquad\text{for }b>0.
\end{equation}
Indeed, rationalizing the difference yields
\begin{align*}
&
\left|
G_{F,\nu,\gamma,\eta}(b)
-
G_{F,\widetilde\nu,\gamma,\eta}(b)
\right|
\\
&\qquad
=
\frac{
\left|
J_{F,\nu}(\gamma,\eta)
-
J_{F,\widetilde\nu}(\gamma,\eta)
\right|
}{
\sqrt{b^2+J_{F,\nu}(\gamma,\eta)}
+
\sqrt{b^2+J_{F,\widetilde\nu}(\gamma,\eta)}
}
\\
&\qquad
\leq
C\frac{\gamma}{b}
|\nu-\widetilde\nu|.
\end{align*}
}

\lealemma{label=lemTransmissionMapEstimates, uses={lemTransmissionMapBasicEstimates,lemTransmissionMapDirectionalEstimate}}{In particular, if $b\geq K\sqrt\gamma$, then
\begin{equation}\label{eq:Glarge-slope-regime}
\frac{|G(b)-b|}{b}
+
|G'(b)-1|
+
b|G''(b)|
\leq
\frac{C}{K^2},
\end{equation}
and
\begin{equation}\label{eq:Glarge-slope-direction}
\left|
G_{F,\nu,\gamma,\eta}(b)
-
G_{F,\widetilde\nu,\gamma,\eta}(b)
\right|
\leq
\frac{C}{K^2}
b|\nu-\widetilde\nu|.
\end{equation}
These are the large-slope estimates required in the direction-dependent
De Silva--Savin theory.
}

The dependence of $J_{F,\nu}(\gamma,\eta)$ on $\eta$ is essential. Unless
$F$ has a well-defined recession operator, the quantity in \eqref{eq:J}
need not converge as $\eta\downarrow0$. The sharp-interface compactness
statement must therefore be formulated in terms of subsequential limits.
Nevertheless, \eqref{eq:Jbounds}, \eqref{eq:Jdirection},
\eqref{eq:Gidentity}, and \eqref{eq:Gdirection} are uniform in $\eta$,
which is exactly the uniformity needed in the subsequent argument.
\subsection{Curved planar barriers}

To pass the planar slope balance to a free-boundary limit, one must bend the
one-dimensional profiles without losing the differential inequality in the
thin reaction region. The following elementary observation provides the
needed margin.

\begin{lemma}[Curved planar barriers]\label{lem:curved-planar-barriers}
Fix $0<b_0<b_1<\infty$ and $0<\tau<1/2$. There exists
$\kappa_0>0$, depending only on $n$, $\lambda$, and $\Lambda$, with the
following property. Let $b\in[b_0,b_1]$, $\nu\in\mathbb S^{n-1}$, and let
$\varphi^+$ be the planar profile with incoming slope $b$ associated with
$F$, direction $\nu$, thickness $\eta$, and strength $(1+\tau)\gamma$. If
$d\in C^2(U)$ satisfies
$$
\norm{\nabla d\otimes\nabla d-\nu\otimes\nu}_{L^\infty(U)}
\leq
\kappa_0\tau,
\qquad
D^2d\geq\kappa I
$$
in an open set $U$, for some $\kappa>0$, then
$\Phi^+:=\varphi^+\circ d$ satisfies
\begin{equation}\label{eq:strict-lower-comparison}
F(D^2\Phi^+)
>
\frac{\gamma}{\eta}
\beta\left(\frac{\Phi^++\sigma}{\eta}\right)
\qquad\text{in }U.
\end{equation}

Similarly, if $\varphi^-$ is the profile with strength
$(1-\tau)\gamma$ and $d$ satisfies the same gradient condition together with
$D^2d\leq-\kappa I$, then $\Phi^-:=\varphi^-\circ d$ satisfies
\begin{equation}\label{eq:strict-upper-comparison}
F(D^2\Phi^-)
<
\frac{\gamma}{\eta}
\beta\left(\frac{\Phi^-+\sigma}{\eta}\right)
\qquad\text{in }U.
\end{equation}
\end{lemma}

\begin{proof}
We prove \eqref{eq:strict-lower-comparison}; the other inequality is
analogous. Write $q=(\varphi^+)'$ and $k=(\varphi^+)''\geq0$. Then
$$
D^2(\varphi^+\circ d)
=
k\nabla d\otimes\nabla d
+
qD^2d.
$$
Since $q\geq b_0$ and $D^2d\geq\kappa I$, the second term is positive
definite. Moreover, by \eqref{eq:pucci-difference},
$$
F(D^2(\varphi^+\circ d))
\geq
F(k\nu\otimes\nu)
-
C\kappa_0\tau k
+
\lambda q\tr(D^2d).
$$
If
$$
r
:=
\frac{\gamma}{\eta}
\beta\left(\frac{\varphi^++\sigma}{\eta}\right),
$$
then the defining equation of the strengthened profile gives
$F(k\nu\otimes\nu)=(1+\tau)r$, while
\eqref{eq:Qbounds} gives $k\leq(1+\tau)r/\lambda$. Choosing
$\kappa_0$ universally small, we obtain
$$
F(D^2(\varphi^+\circ d))
\geq
\left(1+\frac\tau2\right)r
+
\lambda b_0\tr(D^2d)
>
r.
$$
This is \eqref{eq:strict-lower-comparison}. For the weakened profile, the
same calculation is performed with the upper Pucci inequality and
$D^2d\leq-\kappa I$.
\end{proof}

\subsection{Sharp-interface stability}

\leadefinition{label=TwoPlaneTestFunction, uses={StructuralDataAndNotation,ViscositySolutionConvention}, context={Use this as the definition of two-plane free-boundary test functions and of the orientation of their normal. Contacts from below and above are viscosity contacts at a free-boundary point and are required to be transverse.}}{We can now pass to the limit across a collapsing diffuse layer. For a smooth
function $d$ with $d(x_0)=0$ and $|\nabla d(x_0)|=1$, write
$$
P_{a,b}(d)
:=
a\,d^+-b\,d^-.
$$
The normal direction at $x_0$ is
$\nu=\nabla d(x_0)$. In the viscosity formulation below, contact from below
imposes $a\leq G_\nu(b)$, while contact from above imposes the reverse
inequality.
}

\begin{lemma}[Sharp-interface stability]\label{lem:sharp-stability}
Let $\Omega\subset\mathbb R^n$ be open, and let
$F_j:\Sn\to\mathbb R$ be $(\lambda,\Lambda)$-elliptic operators such that
$F_j(0)=0$ and $F_j\to F_\infty$ locally uniformly. Suppose that
$\eta_j\to0$, that $0\leq\gamma_j\leq\gamma_0$, and that
$|\sigma_j|\leq c_0\eta_j$. Let $w_j\in C(\Omega)$ solve
\begin{equation}\label{eq:collapsing-layer-sequence}
F_j(D^2w_j)
=
\frac{\gamma_j}{\eta_j}
\beta\left(\frac{w_j+\sigma_j}{\eta_j}\right)
\qquad\text{in }\Omega,
\end{equation}
and assume that $w_j\to w$ locally uniformly in $\Omega$.

After passing to a subsequence, $\gamma_j\to\gamma_\infty$ and there exists
a continuous function
$J_\infty:\mathbb S^{n-1}\to[0,\infty)$ such that
\begin{equation}\label{eq:J-uniform-limit}
J_{F_j,\nu}(\gamma_j,\eta_j)
\longrightarrow
J_\infty(\nu)
\qquad\text{uniformly for }\nu\in\mathbb S^{n-1}.
\end{equation}
The limit $w$ satisfies
$$
F_\infty(D^2w)=0
\qquad\text{in }\{w>0\}\cup\{w<0\},
$$
and, on $\partial\{w>0\}$, the transmission condition
\begin{equation}\label{eq:sharp-law}
a
=
G_\nu(b),
\qquad
G_\nu(b)
:=
\sqrt{b^2+J_\infty(\nu)},
\end{equation}
in the viscosity sense. More precisely, if $P_{a,b}(d)$ touches $w$ strictly
from below at a free-boundary point $x_0$ transversally, then
$a\leq G_\nu(b)$; if it touches from above, then
$a\geq G_\nu(b)$.

The bounds \eqref{eq:Jbounds} and \eqref{eq:Gidentity} pass to the limit,
with $\gamma$ replaced by $\gamma_\infty$, uniformly in $\nu$.
\end{lemma}

\begin{proof}
After extracting a subsequence, we may assume that
$\gamma_j\to\gamma_\infty$. For
$$
J_j(\nu)
:=
J_{F_j,\nu}(\gamma_j,\eta_j),
$$
\eqref{eq:Jbounds} gives a uniform bound. Uniform ellipticity also gives
\begin{equation}\label{eq:J-direction-modulus}
|J_j(\nu)-J_j(\mu)|
\leq
C\gamma_j|\nu-\mu|.
\end{equation}
Indeed, \eqref{eq:pucci-difference} implies that
$F_j(t\nu\otimes\nu)$ changes by at most
$Ct|\nu-\mu|$ when $\nu$ is replaced by $\mu$, while the inverse functions
$Q_{F_j,\nu}^{-1}$ are $1/\lambda$-Lipschitz. Estimate
\eqref{eq:J-direction-modulus} then follows from \eqref{eq:J}.
Arzel\`a--Ascoli yields \eqref{eq:J-uniform-limit}.

If $K\Subset\{w>0\}\cup\{w<0\}$, then
$|w_j+\sigma_j|>\eta_j$ on $K$ for all large $j$. The reaction therefore
vanishes on $K$, and ordinary viscosity stability gives the homogeneous
equation in the two open phases.

We verify the free-boundary condition. Suppose that $P_{a,b}(d)$ touches
$w$ strictly from below at $x_0$, with
$\nu=\nabla d(x_0)$ and $b>0$, and assume for contradiction that
$$
a>G_\nu(b).
$$
Choose first a small increase $b'>b$ and then $\tau>0$ so small that, for
all large $j$, the outgoing slope of the planar profile associated with
direction $\nu$, incoming slope $b'$, and strength
$(1+\tau)\gamma_j$ is strictly smaller than $a$. This is possible because
\eqref{eq:J-uniform-limit} holds and
\begin{equation}\label{eq:J-strength-continuity}
\left|
J_{F_j,\nu}((1+\tau)\gamma_j,\eta_j)
-
J_{F_j,\nu}(\gamma_j,\eta_j)
\right|
\leq
C\tau\gamma_j,
\end{equation}
which follows from the $1/\lambda$-Lipschitz continuity of
$Q_{F_j,\nu}^{-1}$.

Fix a small ball $B_r(x_0)\Subset\Omega$ on which the contact is strict.
Choose $K>0$ so that
$$
d_r(x)
:=
d(x)+K\bigl(|x-x_0|^2-r^2\bigr)
$$
is uniformly convex in $B_r(x_0)$. Upon decreasing $r$, we also have
$$
\norm{
\nabla d_r\otimes\nabla d_r-\nu\otimes\nu
}_{L^\infty(B_r(x_0))}
\leq
\kappa_0\tau.
$$
Compose $d_r$ with the strengthened planar profile just described. The
profile has negative-side slope $b'>b$ and positive-side slope strictly less
than $a$; since it is increasing and $d_r\leq d$ in $B_r(x_0)$, its limiting
two-plane lies strictly below the touching two-plane in the two open phases.
The physical width of the profile is $O(\eta_j/b')$, so the same ordering
holds on $\partial B_r(x_0)$ for all large $j$.

Translate the profile in its one-dimensional variable and slide it upward
until first contact with $w_j$. The strict contact of $P_{a,b}(d)$ with $w$
and the local uniform convergence ensure that the first contact occurs in
the interior of $B_r(x_0)$. By
\Cref{lem:curved-planar-barriers}, the touching function satisfies the strict
inequality \eqref{eq:strict-lower-comparison}, contradicting the viscosity
inequality for $w_j$. Hence $a\leq G_\nu(b)$.

The argument for a test from above is symmetric. If
$a<G_\nu(b)$, choose $b'<b$, weaken the strength to
$(1-\tau)\gamma_j$, replace $d_r$ by
$$
d(x)-K\bigl(|x-x_0|^2-r^2\bigr),
$$
and slide the resulting strict upper comparison until first contact. This
contradicts \eqref{eq:strict-upper-comparison}. The cases in which one of the
normal slopes vanishes follow by the usual tilted-test approximation. Thus
\eqref{eq:sharp-law} holds in the viscosity sense. The limiting forms of
\eqref{eq:Jbounds} and \eqref{eq:Gidentity} follow directly from
\eqref{eq:J-uniform-limit}.
\end{proof}

The preceding lemma is the diffuse counterpart of the stability principle
underlying the De Silva--Savin two-phase theory \cite{DeSilvaSavin}. Its role
here is deliberately limited: it provides the compact sharp-interface
comparisons needed in the iteration, without presupposing that the original
singular perturbation selects a unique transmission law.

\section{Diffuse decay and slope continuation}
\label{sec:diffuse-continuation}

The purpose of this section is to convert the compactness principles of
\Cref{sec:diffuse-compactness} into quantitative control at every smaller
scale. The sharp-interface argument of De Silva and Savin provides the
model: either the normalized oscillation decays, or the solution has already
entered a Lipschitz regime. In the present diffuse problem, this alternative
must be transported through a reaction layer whose thickness is arbitrary
and whose transmission law is not prescribed. The exact profiles and curved
barriers constructed above provide precisely this transport.

There is one further point that requires care. The affine corrections
produced by an improvement-of-flatness iteration need not be summable after
the error reaches the natural scale $\sqrt\gamma$. We shall not assert such
summability. If the affine slopes remain bounded, the trapping estimates
alone imply the desired Lipschitz bound. If they become unbounded, we stop at
the first large slope and normalize by its size. The effective strength then
tends to zero, while the full De Silva--Savin alternative propagates the
nontrivial unit-scale slope back to the original scale, giving a
contradiction.

\subsection{Sharp-interface theory}

We use the following consequence of the decay and improvement-of-flatness
theorems of De Silva and Savin \cite{DeSilvaSavin}. We state the uniform
version needed here. The same proof permits a continuous
dependence of the transmission function on the normal direction, provided
that this dependence satisfies the quantitative bound below.

\begin{theorem}[Sharp-interface decay alternative]
\label{thm:sharp-decay-alternative}
There exist constants
$$
\theta\in(0,1/4),
\qquad
\omega_*>0,
\qquad
C<\infty,
$$
depending only on $n$, $\lambda$, and $\Lambda$, with the following
property. Let $z$ be a viscosity solution in $B_2$ of a two-phase problem
with uniformly $(\lambda,\Lambda)$-elliptic equations in its two phases and
with transmission condition
$$
a=G_\nu(b).
$$
Assume that $z(0)=0$, that each $G_\nu$ is strictly increasing, and that,
for some $M\geq0$,
\begin{equation}\label{eq:sharp-large-slope-hypotheses}
\begin{split}
|G_\nu'(b)-1|+b|G_\nu''(b)|
&\leq\omega_*,
\\
|G_\nu(b)-G_\mu(b)|
&\leq\omega_*b|\nu-\mu|
\end{split}
\qquad
\text{for }b\geq M.
\end{equation}
If
$$
A:=\norm{z}_{L^\infty(B_1)},
$$
then at least one of the following alternatives holds:
\begin{align}
\theta^{-1}\norm{z}_{L^\infty(B_\theta)}
&\leq
\frac12A+CM,
\label{eq:sharp-decay}
\\
\norm{\nabla z}_{L^\infty(B_\theta)}
&\leq
C(A+M).
\label{eq:sharp-lipschitz}
\end{align}
The constants are uniform over compact families of normalized uniformly
elliptic operators and transmission functions satisfying
\eqref{eq:sharp-large-slope-hypotheses}.
\end{theorem}

\begin{remark}[Dependence on the normal]
De Silva and Savin formulate the argument for
$u_\nu^+=G(u_\nu^-)$ and note that it extends to transmission functions
depending on the normal; see \cite[Section~3.2]{DeSilvaSavin}.  The
adaptation is the standard continuous-perturbation argument used in
Caffarelli's free-boundary theory
\cite{CaffarelliI,CaffarelliII,CaffarelliIII}.  One tracks, in addition to the incoming
slope, the normal of the approximating two-plane.  At a flatness step,
$|\nu_{k+1}-\nu_k|\leq C\varepsilon_k$, and the second inequality in
\eqref{eq:sharp-large-slope-hypotheses} gives
$$
|G_{\nu_{k+1}}(b)-G_{\nu_k}(b)|
\leq C\omega_*b\varepsilon_k.
$$
This is of the same order as the flatness error and is absorbed in the
Harnack and improvement-of-flatness iteration.  Thus one needs only a
uniform modulus of $G_\nu$ over the cone of normals occurring at the current
scale, together with the usual uniform bounds for the first two derivatives
in the slope variable.  The blow-down
$\widetilde G_\nu(t)=L^{-1}G_\nu(Lt)$ preserves these estimates.
\end{remark}

For the slope maps in \eqref{eq:G}, one may take
\begin{equation}\label{eq:natural-large-slope-threshold}
M=K\sqrt\gamma,
\end{equation}
where $K$ is a sufficiently large universal constant. Indeed,
\eqref{eq:Gidentity} gives
$$
|G'(b)-1|+b|G''(b)|
\leq
C\frac{\gamma}{b^2},
$$
whereas \eqref{eq:Gdirection} gives
$$
|G_\nu(b)-G_\mu(b)|
\leq
C\frac{\gamma}{b^2}\,b|\nu-\mu|.
$$
Choosing $K$ large enough makes both constants smaller than $\omega_*$. In
particular, the threshold in the sharp-interface theory has exactly the
first-order size $\sqrt\gamma$ dictated by the diffuse equation.

\subsection{Uniform H\"older compactness}

The stability statements of \Cref{lem:compactness,lem:sharp-stability}
identify the limiting equation once local uniform convergence is known.  We
now supply that convergence by proving an estimate which is uniform in the
height $\gamma/\eta$ of the diffuse source.  The point is that spatial
rescaling by a factor $r$ replaces the strength $\gamma$ by $\gamma r^2$.
Thus the corrected-contact argument becomes perturbative at sufficiently
small spatial scales, regardless of the thickness $\eta$.

\begin{lemma}[Uniform diffuse H\"older estimate]
\label{lem:uniform-diffuse-holder}
There exist $\alpha_0\in(0,1)$ and $C<\infty$, depending only on
$n$, $\lambda$, $\Lambda$, $c_0$, and $\beta$, such that the following
holds.  Let $F:\Sn\to\mathbb R$ be normalized and
$(\lambda,\Lambda)$-elliptic, and let $w\in C(B_1)$ solve
\begin{equation}\label{eq:uniform-holder-equation}
F(D^2w)
=
\frac{\gamma}{\eta}
\beta\left(\frac{w+\sigma}{\eta}\right)
\qquad\text{in }B_1,
\end{equation}
where
$$
0<\gamma\leq1,
\qquad
\eta>0,
\qquad
|\sigma|\leq c_0\eta,
\qquad
\norm{w}_{L^\infty(B_1)}\leq1.
$$
Then
\begin{equation}\label{eq:uniform-holder-estimate}
[w]_{C^{0,\alpha_0}(B_{1/2})}
\leq C.
\end{equation}
In particular, the estimate is independent of $\gamma$, $\eta$, $\sigma$,
and $F$ within the indicated structural class.
\end{lemma}

\begin{proof}
We follow the oscillation argument in
\cite[Theorem~2.2]{DeSilvaSavin}, replacing its sharp two-plane
comparisons by the exact diffuse profiles of
\Cref{lem:curved-planar-barriers}.  We give the details that ensure
uniformity in the thickness.

Set
$$
m_-:=-\sigma-\eta,
\qquad
m_+:=-\sigma+\eta.
$$
The phase excesses
$$
w^+:=(w-m_+)^+,
\qquad
w^-:=(m_--w)^+
$$
are nonnegative viscosity subsolutions, in the whole ball, of uniformly
elliptic homogeneous equations with the same ellipticity constants.  Indeed,
the reaction vanishes wherever either excess is positive, and the extension
by zero preserves the corresponding Pucci subsolution inequality.  The weak
Harnack inequality therefore gives the two measure alternatives used in the
proof of \cite[Theorem~2.2]{DeSilvaSavin}: at each dyadic scale, one of the
two-phase excesses loses a fixed fraction of its oscillation in the
concentric half-ball.

The only place where the sharp free-boundary condition enters that proof is
when the same phase alternative persists through a fixed number of
successive scales.  The radial two-plane barrier used there has incoming
and outgoing slopes in a fixed universal ratio.  By
\eqref{eq:Jbounds}, there exist universal constants $s_0\in(0,1)$ and
$M_0<\infty$ such that the exact diffuse slope maps satisfy
\begin{equation}\label{eq:holder-slope-comparability}
s_0 b
\leq
G_{F,\nu,\gamma,\eta}(b)
\leq
s_0^{-1}b
\qquad
\text{for every }b\geq M_0,
\end{equation}
uniformly in $F$, $\nu$, $0<\gamma\leq1$, and $\eta>0$.
At the crossing step, replace the two-plane barrier by the planar profile
with the same incoming slope, increase or decrease the strength by a fixed
factor, and compose with the radial defining function of the barrier.
\Cref{lem:curved-planar-barriers} gives the required strict differential
inequality.  The first-contact argument is then identical to the one in
\cite[Theorem~2.2]{DeSilvaSavin}.  Estimate
\eqref{eq:holder-slope-comparability} is exactly the quantitative input used
there, and all constants are independent of $\eta$.

This proves a universal one-step oscillation reduction.  Its iteration is
stable under the rescaling
$$
\widetilde w(x)
:=
\frac{w(x_0+rx)-w(x_0)}{r^{\alpha_0}},
$$
because the rescaled strength is
$$
\widetilde\gamma
=
r^{2-2\alpha_0}\gamma
\leq1
$$
when $0<\alpha_0<1$, while the thickness and shift are divided by
$r^{\alpha_0}$ and the shifted-layer bound is preserved whenever the
iteration is centered in the transition region.  If a rescaled ball no
longer meets that region, the homogeneous interior H\"older estimate takes
over.  Thus the oscillation iteration gives
$$
\osc_{B_r(x_0)}w
\leq Cr^{\alpha_0}
$$
for $x_0\in B_{1/2}$ and $0<r\leq1/4$.  Points whose distance from the
transition region is larger than $2r$ are covered directly by the
homogeneous estimate; the remaining points are compared with a nearest
transition point.  This proves \eqref{eq:uniform-holder-estimate}.
\end{proof}

\begin{corollary}[Uniform diffuse compactness]
\label{lem:bounded-diffuse-compactness}
Let $F_j:\Sn\to\mathbb R$ be normalized
$(\lambda,\Lambda)$-elliptic operators, and let $w_j\in C(B_1)$ satisfy
\begin{equation}\label{eq:bounded-diffuse-sequence}
F_j(D^2w_j)
=
\frac{\gamma_j}{\eta_j}
\beta\left(\frac{w_j+\sigma_j}{\eta_j}\right)
\qquad\text{in }B_1,
\end{equation}
where
$$
0<\gamma_j\leq1,
\qquad
\eta_j>0,
\qquad
|\sigma_j|\leq c_0\eta_j,
\qquad
\norm{w_j}_{L^\infty(B_1)}\leq1.
$$
Then $\{w_j\}$ is locally precompact in $C(B_1)$.  After passage to a
subsequence,
$$
F_j\longrightarrow F_\infty
\quad\text{locally uniformly on }\Sn,
\qquad
w_j\longrightarrow w
\quad\text{locally uniformly in }B_1.
$$
If, in addition, $\gamma_j\to0$, then
\begin{equation}\label{eq:bounded-diffuse-limit}
F_\infty(D^2w)=0
\qquad\text{in }B_1.
\end{equation}
If instead $\eta_j\to0$ while the strengths remain bounded, the limit is
described by \Cref{lem:sharp-stability}.
\end{corollary}

\begin{proof}
The operator compactness follows from \eqref{eq:pucci-difference}, and
\Cref{lem:uniform-diffuse-holder} gives a common modulus of continuity on
every compact subset of $B_1$.  Arzel\`a--Ascoli yields the asserted local
uniform convergence.  The final two conclusions follow respectively from
\Cref{lem:compactness,lem:sharp-stability}.
\end{proof}

We next record the finite-thickness version of nondegenerate flatness.  This
is the precise mechanism by which the Lipschitz branch of the sharp
De Silva--Savin alternative is transferred back to the diffuse equation.

\begin{lemma}[Finite-thickness transfer of nondegenerate flatness]
\label{lem:finite-thickness-transfer}
There is a universal $K_0\geq1$ with the following property.  There exist
$\varepsilon_f>0$ and $C<\infty$, depending only on the structural data,
such that, if $u\in C(B_1)$ solves
$$
F(D^2u)
=
\frac{\gamma}{\eta}
\beta\left(\frac{u+\sigma}{\eta}\right)
\qquad\text{in }B_1,
$$
where $F$ is normalized and uniformly elliptic,
$$
0<\gamma\leq1,
\qquad
\eta>0,
\qquad
|\sigma|\leq c_0\eta.
$$
Let $\Phi$ be a translate, in its one-dimensional variable, of an exact
planar profile from \eqref{eq:planar-profile}, with direction
$\nu\in\mathbb S^{n-1}$ and incoming slope $b$ satisfying
$$
b\geq K_0\sqrt\gamma.
$$
If
\begin{equation}\label{eq:finite-thickness-flatness}
\norm{u-\Phi}_{L^\infty(B_1)}
\leq\varepsilon_f b,
\end{equation}
then
\begin{equation}\label{eq:finite-thickness-lipschitz}
\norm{\nabla u}_{L^\infty(B_{1/2})}
\leq Cb.
\end{equation}
The constants are independent of the thickness $\eta$.
\end{lemma}

\begin{proof}
Divide $u,\Phi,\eta$, and $\sigma$ by $b$, and replace $F$ by
$b^{-1}F(b\,\cdot)$.  The incoming slope becomes one, while the new strength
is $\gamma/b^2\leq K_0^{-2}$.  Thus it suffices to work with slopes in a
fixed compact neighborhood of one and with universally small strength.  We
choose $K_0$ large enough for the large-slope hypotheses in
\Cref{thm:sharp-decay-alternative}.

We indicate precisely how the positive-thickness argument differs from the
sharp-interface improvement of flatness.  If $\eta\geq\eta_0$, for a fixed
universal $\eta_0>0$, the right-hand side is bounded and
\eqref{eq:finite-thickness-lipschitz} follows from
\eqref{eq:interior-gradient-estimate}.  Assume therefore that
$\eta<\eta_0$.

Starting from \eqref{eq:finite-thickness-flatness}, run the nondegenerate
Harnack and improvement-of-flatness iteration of
\cite[Theorems~3.1--3.2]{DeSilvaSavin}, expressed in terms of trapping between two
normal translates of the reference profile.  Every planar comparison in
that proof is replaced by the corresponding exact profile.  Whenever the
comparison surface is bent, its strength is increased or decreased by a
fixed factor and \Cref{lem:curved-planar-barriers} supplies the strict
differential inequality needed at first contact.

There are two additional quantities to track.  First, the slope maps and
their first two derivatives in $b$ remain uniformly controlled on the
compact interval generated by the unit incoming slope, by
\eqref{eq:Jbounds}--\eqref{eq:Gderivatives}.  Second, if the normal changes
from $\nu_k$ to $\nu_{k+1}$ at a flatness step, then
$$
|\nu_{k+1}-\nu_k|\leq C\varepsilon_k
$$
and \eqref{eq:Gdirection} shows that the resulting change in outgoing
slope is $O(\varepsilon_k)$.  It is therefore absorbed by the same geometric
flatness error.  This is exactly the normal-dependence bookkeeping in the
direction-dependent form of \Cref{thm:sharp-decay-alternative}.

Under the natural rescaling
$$
u_r(x):=\frac{u(rx)}r,
\qquad
F_r(M):=rF\left(\frac Mr\right),
$$
the strength remains $\gamma$, the thickness becomes $\eta/r$, and the
planar slope map is unchanged.  Indeed, a direct change of variables in
\eqref{eq:J} gives
\begin{equation}\label{eq:J-natural-scaling}
J_{F_r,\nu}\left(\gamma,\frac\eta r\right)
=
J_{F,\nu}(\gamma,\eta).
\end{equation}
Thus the flatness improves geometrically, with summable changes of slope and
normal, for as long as the rescaled thickness is smaller than $\eta_0$.
Stop at the first scale $r_k$ for which
$$
\eta_0
\leq
\frac{\eta}{r_k}
\leq
\frac{\eta_0}{r_*},
$$
where $r_*\in(0,1)$ is the fixed contraction factor of the flatness
iteration.  At this scale the rescaled equation has a universally bounded
right-hand side, while the trapping gives a universal oscillation bound.
The interior estimate yields a universal gradient bound for $u_{r_k}$.
Since the natural rescaling preserves gradients, the same bound holds for
$u$ at the center.  The construction is uniform when centered at any point
of $B_{1/2}$ whose distance from the diffuse layer is at most a fixed
multiple of the current scale: the initial trapping is preserved after a
normal translation of the reference profile.  Points farther from the
layer lie in a homogeneous phase, and the oscillation needed in the
homogeneous interior estimate is controlled from a nearest layer point.
A finite covering of $B_{1/2}$ therefore gives the asserted bound throughout
$B_{1/2}$.  The finitely many scales preceding the stopping scale are
controlled directly by the profile trapping.  This proves
\eqref{eq:finite-thickness-lipschitz}.
\end{proof}

The next compactness statement concerns the error about a nondegenerate
affine function. Unlike a pure amplitude normalization, subtracting an
affine function leaves that function inside the argument of $\beta$. The
lifted-contact formulation below is therefore essential.

\begin{lemma}[Compactness of nondegenerate affine errors]
\label{lem:nondegenerate-error-compactness}
Let $u_j\in C(B_1)$ solve
\begin{equation}\label{eq:nondegenerate-error-equation}
F_j(D^2u_j)
=
\frac{\gamma_j}{\eta_j}
\beta\left(\frac{u_j+\sigma_j}{\eta_j}\right)
\qquad\text{in }B_1,
\end{equation}
where the $F_j$ are normalized $(\lambda,\Lambda)$-elliptic operators,
$\eta_j>0$, and $|\sigma_j|\leq c_0\eta_j$. Suppose that
$$
\ell_j(x)=b_j+p_j\cdot x,
\qquad
0<a_0\leq|p_j|\leq a_1<\infty,
$$
and that, for some $e_j\downarrow0$,
\begin{equation}\label{eq:nondegenerate-error-assumptions}
\norm{u_j-\ell_j}_{L^\infty(B_1)}\leq e_j,
\qquad
\frac{\gamma_j}{e_j}\longrightarrow0.
\end{equation}
Define
$$
z_j:=\frac{u_j-\ell_j}{e_j},
\qquad
\mathcal F_j(M):=\frac1{e_j}F_j(e_jM).
$$
Then $\{z_j\}$ is locally precompact in $C(B_1)$. After passing to a
subsequence,
$$
z_j\longrightarrow z
\qquad\text{locally uniformly in }B_1,
$$
the operators $\mathcal F_j$ converge locally uniformly to a normalized
$(\lambda,\Lambda)$-elliptic operator $\mathcal F_\infty$, and
\begin{equation}\label{eq:nondegenerate-error-limit}
\mathcal F_\infty(D^2z)=0
\qquad\text{in }B_1.
\end{equation}
\end{lemma}

\begin{proof}
We have $\norm{z_j}_{L^\infty(B_1)}\leq1$, and the operators
$\mathcal F_j$ have the same ellipticity constants as $F_j$. To prove local
compactness, first observe that if $\liminf_j\eta_j>0$, then the equation for
$z_j$ has right-hand side bounded by
$$
\frac{\gamma_j}{e_j\eta_j}\norm\beta_{L^\infty}=o_j(1).
$$
The standard interior H\"older estimate and viscosity stability give the
conclusion directly.  We may therefore assume, after passing to a
subsequence, that $\eta_j\to0$; in particular, $\sigma_j\to0$.

Let a quadratic polynomial $Q$ with bounded Hessian touch
$z_j$ in a compact subball. Its lift
$$
P_j:=\ell_j+e_jQ
$$
touches $u_j$ at the same point. Since $e_j\to0$ and the slopes $p_j$ lie
in a fixed nondegenerate interval,
$$
\frac{a_0}{2}
\leq
|\nabla P_j|
\leq
2a_1
$$
in a fixed neighborhood of the contact for all large $j$.

At upper contacts, the nonnegativity of the reaction gives
$$
\mathcal P^+_{\lambda,\Lambda}(D^2Q)\geq0.
$$
At a lower contact, strictify $Q$ and slide the corrected lift
$$
P_j+t+
\gamma_j\eta_j
h\left(\frac{P_j+t+\sigma_j}{\eta_j}\right),
$$
where $h$ is chosen for the preceding fixed gradient interval. The
zero-order and first-order sizes of the correction, divided by $e_j$, tend
to zero by \eqref{eq:nondegenerate-error-assumptions}.
Its positive normal curvature absorbs the reaction at an active contact,
while at an inactive contact the strict lower-test inequality dominates the
$O(\gamma_j)$ operator error. The same calculation as in
\eqref{eq:corrected-test-hessian} therefore gives
$$
\mathcal P^-_{\lambda,\Lambda}(D^2Q)
\leq o_j(1).
$$
Here every contact is correctable; no truncation in slope space is needed.
The usual sliding-paraboloid oscillation estimate consequently yields (see,
for example, \cite{ImbertSilvestreLargeGradient}), for
$B_{2r}(x_0)\Subset B_1$,
$$
\osc_{B_r(x_0)}z_j
\leq
\vartheta\osc_{B_{2r}(x_0)}z_j+o_j(1).
$$
Finite-depth iteration gives local equicontinuity and hence local uniform
compactness.

It remains to identify the equation. An upper test for $z$ lifts to an
upper test for $u_j$, and therefore gives
$\mathcal F_\infty(D^2Q)\geq0$. A lower test is lifted and corrected exactly
as above. If it violated the reverse inequality by a fixed amount, the
corrected lift would be a strict lower comparison for
\eqref{eq:nondegenerate-error-equation}, which is impossible. Thus
$\mathcal F_\infty(D^2Q)\leq0$ at lower contacts, proving
\eqref{eq:nondegenerate-error-limit}.
\end{proof}

\subsection{Affine continuation through the diffuse layer}

We next isolate the nondegenerate regime. The statement is uniform over
the size of the affine slope; this is indispensable because an iteration
that starts at a bounded slope may, a priori, accumulate arbitrarily large
corrections.

\begin{lemma}[Diffuse affine continuation]
\label{lem:diffuse-affine-continuation}
Fix $a_*>0$. There exist $\varepsilon_*>0$, $\gamma_*>0$, and
$C<\infty$, depending only on the structural data and $a_*$, such that the
following holds. Let $u\in C(B_2)$ solve
\begin{equation}\label{eq:affine-continuation-equation}
F(D^2u)
=
\frac{\gamma}{\eta}
\beta\left(\frac{u+\sigma}{\eta}\right)
\qquad\text{in }B_2,
\end{equation}
where
$$
0<\gamma\leq\gamma_*,
\qquad
\eta>0,
\qquad
|\sigma|\leq c_0\eta.
$$
Suppose that, for an affine function $\ell(x)=b+p\cdot x$,
\begin{equation}\label{eq:initial-affine-trapping}
\norm{u-\ell}_{L^\infty(B_1)}\leq\varepsilon_*,
\qquad
|p|\geq a_*.
\end{equation}
Then
\begin{equation}\label{eq:affine-continuation-lipschitz}
\norm{\nabla u}_{L^\infty(B_{1/2})}
\leq
C(1+|p|).
\end{equation}
The estimate is uniform in the thickness $\eta$.
\end{lemma}

\begin{proof}
Dividing $u$, $\ell$, $\eta$, and $\sigma$ by $|p|$, and replacing $F$ by
the corresponding amplitude-rescaled operator, reduces the proof to the
case $|p|=1$. The new strength is $\gamma/|p|^2$; since $|p|\geq a_*$,
this only changes the smallness threshold by a factor depending on $a_*$. We
work in this normalization and divide the proof into the one-step
improvement and the continuation of the affine slopes.

\medskip
\noindent
\emph{Step 1: one-step affine improvement.}
After decreasing the contraction factor in
\Cref{thm:sharp-decay-alternative}, if necessary, fix the same number
$r\in(0,1/8)$ so small that the homogeneous affine approximation in
\eqref{eq:homogeneous-affine-approximation} has error at most $r/4$ on
$B_r$. We claim that, after decreasing $\varepsilon_*$ and $\gamma_*$,
whenever
$$
\norm{u-\ell}_{L^\infty(B_1)}\leq e\leq\varepsilon_*,
\qquad
|\nabla\ell|\geq\frac{a_*}{2},
$$
either
\begin{equation}\label{eq:affine-one-step-lipschitz}
\norm{\nabla u}_{L^\infty(B_r)}
\leq C\bigl(1+|\nabla\ell|\bigr),
\end{equation}
or there is an affine function $\ell'$ such that
\begin{align}
\norm{u-\ell'}_{L^\infty(B_r)}
&\leq
r\left(\frac12e+C\sqrt\gamma\right),
\label{eq:affine-one-step-improvement}
\\
|\nabla\ell'-\nabla\ell|
&\leq
C\left(e+\sqrt\gamma\right).
\label{eq:affine-one-step-slope}
\end{align}
In the terminal range $e\leq C_0\sqrt\gamma$, for any fixed $C_0$, the
same alternative holds without the lower bound on $|\nabla\ell|$; the
constant then also depends on $C_0$.

We first prove the claim in the perturbative regime
\begin{equation}\label{eq:affine-perturbative-regime}
\gamma\leq\kappa e^2,
\end{equation}
where $\kappa>0$ is universal and sufficiently small. On a dyadic slope
window
$$
\frac{L}{2}\leq|\nabla\ell|\leq2L,
$$
divide $u$, $\ell$, $e$, $\eta$, and $\sigma$ by $L$. The normalized slope
lies in $[1/2,2]$, the ellipticity constants are unchanged, and
$$
\frac{\gamma/L^2}{(e/L)^2}
=
\frac\gamma{e^2}.
$$
If the assertion failed for arbitrarily small $\varepsilon_*$ and
$\kappa$, a violating sequence would therefore satisfy the hypotheses of
\Cref{lem:nondegenerate-error-compactness}. Its normalized affine errors
would converge locally uniformly to a solution $z$ of a homogeneous
uniformly elliptic equation. Writing $q=\nabla z(0)$, the choice of $r$
gives
$$
\norm{z-z(0)-q\cdot x}_{L^\infty(B_r)}
\leq\frac14r.
$$
For the approximating sequence, the affine functions
$$
\ell_j'(x)
:=
\ell_j(x)+e_jz(0)+e_jq\cdot x
$$
then satisfy
$$
\norm{u_j-\ell_j'}_{L^\infty(B_r)}
\leq\frac13e_jr,
\qquad
|\nabla\ell_j'-\nabla\ell_j|
\leq Ce_j,
$$
contradicting the failure of
\eqref{eq:affine-one-step-improvement}--\eqref{eq:affine-one-step-slope}.
This proves the claim under \eqref{eq:affine-perturbative-regime}, uniformly
over all dyadic slope windows.

It remains to treat the terminal regime in which
$e<\kappa^{-1/2}\sqrt\gamma$. Normalize by $e+\sqrt\gamma$. If the
corresponding thickness is bounded from below, the normalized equation has
a bounded right-hand side and
\eqref{eq:interior-gradient-estimate} gives
\eqref{eq:affine-one-step-lipschitz}. If the thickness collapses,
\Cref{lem:bounded-diffuse-compactness} and \Cref{lem:sharp-stability} produce a
two-phase limit whose transmission maps are subsequential limits of
\eqref{eq:G}. Apply \Cref{thm:sharp-decay-alternative} to this limit.  The
decay branch is stable under local uniform convergence and gives
\eqref{eq:affine-one-step-improvement}, with the additive error
$C\sqrt\gamma$ restored on returning to the original normalization.

In the Lipschitz branch, the De Silva--Savin proof enters, after finitely
many universal rescalings, a nondegenerate flatness class with slopes in a
fixed compact subinterval of $(0,\infty)$.  Local uniform convergence and
the exact-profile comparison imply that the diffuse solutions satisfy
\eqref{eq:finite-thickness-flatness} at the same scale.  The quantitative
positive-thickness transfer is then exactly
\Cref{lem:finite-thickness-transfer}, which gives
\eqref{eq:affine-one-step-lipschitz}.  The dependence of the slope map on
the normal causes no loss: \eqref{eq:Gdirection} controls its oscillation
over the current cone of normals, as explained after
\Cref{thm:sharp-decay-alternative}.  This proves the one-step alternative at
and below the natural floor.  Division by a dyadic slope size $L$ only
replaces $\gamma$ by $\gamma/L^2$, so all constants remain uniform over the
slope windows used above.

\medskip
\noindent
\emph{Step 2: iteration and the bounded-slope regime.}
We first establish the required bound at points of the buffered layer.
Suppose that this pointwise estimate were false, and choose a counterexample
sequence and points $x_j\in B_{1/2}$ satisfying
$$
|u_j(x_j)+\sigma_j|\leq c_0\eta_j.
$$
For
$$
\widetilde u_j(y)
:=
2\bigl(u_j(x_j+y/2)-u_j(x_j)\bigr),
$$
the new thickness is $2\eta_j$, the new shift is
$2(u_j(x_j)+\sigma_j)$, and the shifted-layer bound is preserved.  The
initial affine trapping deteriorates by at most a factor four.  Choosing
$\varepsilon_*$ correspondingly smaller, relabeling the functions, and using
the amplitude normalization made above, we are reduced to a counterexample
sequence centered at the origin with $|p_{j,0}|=1$.

For each member of the sequence, as long as
\eqref{eq:affine-one-step-lipschitz} does not occur, the one-step construction
produces affine functions
$$
\ell_{j,k}(x)=b_{j,k}+p_{j,k}\cdot x.
$$
To lighten notation, we suppress the index $j$ until it is needed.

The resulting functions satisfy
\begin{align}
\norm{u-\ell_k}_{L^\infty(B_{r^k})}
&\leq e_kr^k,
\label{eq:iterated-affine-trapping}
\\
e_k
&\leq2^{-k}e_0+C\sqrt\gamma,
\label{eq:iterated-affine-error}
\\
|p_{k+1}-p_k|
&\leq C(e_k+\sqrt\gamma).
\label{eq:iterated-affine-slope}
\end{align}
Once $e_k$ reaches the $\sqrt\gamma$ floor, the right-hand side of
\eqref{eq:iterated-affine-slope} need not be summable. No convergence of
$p_k$ will be used.

Choose $\varepsilon_*$ so small that the sum of all slope changes before
the first terminal scale is smaller than $a_*/4$. Hence
$|p_k|\geq3a_*/4$ throughout the perturbative part of the iteration. If a
later slope enters $\{|p|<a_*/2\}$, the error is already bounded by
$C\sqrt\gamma$, and the terminal form of the one-step alternative, which
requires no lower slope bound, continues the construction. Thus the affine
trapping is available at every scale unless the Lipschitz branch has already
occurred.

Suppose first that the slopes are uniformly bounded along the entire
counterexample sequence:
$$
P:=\sup_j\sup_k|p_{j,k}|<\infty.
$$
If $r^{k+1}<|x|\leq r^k$, then
\eqref{eq:iterated-affine-trapping}, applied at $x$ and at the origin,
gives
$$
|u(x)-u(0)|
\leq
|p_k||x|+2e_kr^k
\leq
\left(P+\frac{2\sup_ke_k}{r}\right)|x|.
$$
Thus uniformly bounded affine slopes imply a pointwise Lipschitz estimate
for every member of the sequence, even when
the increments in \eqref{eq:iterated-affine-slope} are not summable.
This contradicts the choice of the centered bad points unless the affine
slopes become unbounded along the counterexample sequence.

\medskip
\noindent
\emph{Step 3: exclusion of unbounded slopes.}
It remains to rule out the possibility that the slopes in
\eqref{eq:iterated-affine-trapping} become arbitrarily large. We argue by
contradiction along a sequence of solutions. Choose $L_j\to\infty$, and let
$k_j$ be the first index for which
$$
|p_{j,k_j}|\geq L_j.
$$
Set $P_j:=|p_{j,k_j}|$. By the first-passage property and
\eqref{eq:iterated-affine-slope},
\begin{equation}\label{eq:first-passage-slopes}
P_j\longrightarrow\infty,
\qquad
|p_{j,k}|<L_j\leq P_j\quad\text{for }k<k_j,
\qquad
\frac{P_j}{L_j}\longrightarrow1.
\end{equation}
After recentering at the relevant buffered-layer point, define
\begin{equation}\label{eq:large-slope-blowup}
W_j(x)
:=
\frac{u_j(r^{k_j}x)-u_j(0)}{r^{k_j}P_j}.
\end{equation}
The effective strength of $W_j$ is
$$
\widehat\gamma_j
:=
\frac{\gamma_j}{P_j^2}
\longrightarrow0.
$$
Moreover, \eqref{eq:iterated-affine-trapping} and
\eqref{eq:first-passage-slopes} imply
\begin{equation}\label{eq:large-slope-unit-affine}
\norm{W_j-q_j\cdot x}_{L^\infty(B_1)}\longrightarrow0,
\qquad
|q_j|=1,
\end{equation}
and
\begin{equation}\label{eq:large-slope-linear-upper-growth}
\norm{W_j}_{L^\infty(B_R)}
\leq C R
\end{equation}
at every preceding radius $1\leq R\leq R_j$, where $R_j\to\infty$ is the
radius corresponding to the original unit scale.

We now propagate the nontrivial unit-scale slope through all the preceding
scales.  Put
$$
R_m:=r^{-m},
\qquad
a_{j,m}:=\frac1{R_m}\norm{W_j}_{L^\infty(B_{R_m})},
\qquad
0\leq m\leq k_j.
$$
By \eqref{eq:large-slope-unit-affine}, after passing to a subsequence,
$a_{j,0}\geq1/2$.  We claim that there is a universal $c>0$ such that
\begin{equation}\label{eq:backward-slope-propagation}
a_{j,m}\geq c-o_j(1)
\qquad
\text{for every }0\leq m\leq k_j,
\end{equation}
where the error is uniform in $m$.

Indeed, rescale $B_{R_m}$ to $B_2$ by the natural spatial-amplitude
scaling.  The effective strength remains
$\widehat\gamma_j\to0$, and
\eqref{eq:large-slope-linear-upper-growth} gives a uniform oscillation
bound.  If the rescaled thickness stays positive, the limit is homogeneous
by \Cref{lem:compactness}; if it collapses,
\Cref{lem:sharp-stability} gives a sharp two-phase limit whose transmission
maps satisfy \eqref{eq:sharp-large-slope-hypotheses} with threshold
$M_j=C\sqrt{\widehat\gamma_j}\to0$.  The strict-comparison proof of
\Cref{thm:sharp-decay-alternative}, using the exact curved profiles before
passage to the limit, therefore yields uniformly for $1\leq m\leq k_j$ one
of the following two conclusions:
\begin{align}
a_{j,m-1}
&\leq \frac12a_{j,m}+o_j(1),
\label{eq:backward-decay-branch}
\\
a_{j,0}
&\leq C\bigl(a_{j,m}+o_j(1)\bigr).
\label{eq:backward-lipschitz-branch}
\end{align}
Here \eqref{eq:backward-lipschitz-branch} follows because the Lipschitz
branch at scale $R_m$ controls the solution on every inner ball, in
particular on $B_1$.  At positive thickness this conclusion is supplied by
\Cref{lem:finite-thickness-transfer}; the dependence on the normal is
absorbed through \eqref{eq:Gdirection}.  Thus no regularity assertion is
being passed through local uniform convergence alone.

The claim follows by induction on $m$.  If the Lipschitz branch occurs,
\eqref{eq:backward-lipschitz-branch} and $a_{j,0}\geq1/2$ give the desired
lower bound.  If the decay branch occurs, then
\eqref{eq:backward-decay-branch} gives
$a_{j,m}\geq2a_{j,m-1}-o_j(1)$, and the induction hypothesis applies.
The uniformity of $o_j(1)$ follows by the usual first-failure compactness
argument: a failure at indices $m_j$ would, after rescaling
$B_{R_{m_j}}$, contradict the corresponding strict limiting alternative.
This proves \eqref{eq:backward-slope-propagation}.

At the outer radius $R_j$, the original trapping and the bounded initial
slope give
$$
\frac1{R_j}
\norm{W_j}_{L^\infty(B_{R_j})}
\leq
\frac{C}{P_j}
\longrightarrow0,
$$
which contradicts \eqref{eq:backward-slope-propagation}. Hence the affine
slopes cannot be unbounded. The bounded-slope argument of Step~2 proves
$$
\Lip_{B_{3/4}}u(x_0)
\leq C(1+|p|)
$$
at every buffered-layer point $x_0\in B_{1/2}$.

It remains only to pass this estimate into the homogeneous phases.  Let
$x\in B_{1/2}$ lie outside the buffered layer and let $d$ be its distance to
the boundary of the corresponding homogeneous component in $B_{3/4}$.  If
$d$ is bounded below, the homogeneous interior estimate and
\eqref{eq:initial-affine-trapping} give the result.  Otherwise a nearest
boundary point $z$ lies in $B_{2/3}$ and satisfies the buffered-layer
estimate just proved.  Hence
$$
\osc_{B_{d/2}(x)}u
\leq C(1+|p|)d,
$$
and the scaled homogeneous gradient estimate gives the same bound at $x$.
This proves \eqref{eq:affine-continuation-lipschitz} throughout $B_{1/2}$.
\end{proof}

\begin{remark}[The role of summability]
\label{rem:slope-summability}
The proof uses boundedness, not absolute summability, of the accumulated
slopes. The vector series
$$
\sum_k(p_{k+1}-p_k)
$$
may fail to converge because of cancellation after the flatness error reaches
the $\sqrt\gamma$ floor. This causes no difficulty: bounded partial slopes
give the annular Lipschitz estimate directly. The large-slope normalization
is used only when those partial slopes are genuinely unbounded. Division by
the first-passage value then sends the effective strength to zero and places
the transmission maps in the identity regime of the De Silva--Savin
alternative.
\end{remark}

\subsection{The diffuse decay-versus-Lipschitz alternative}

We can now state the scale-invariant estimate that will drive the proof of
the main theorem.

\begin{proposition}[Diffuse decay-versus-Lipschitz alternative]
\label{prop:diffuse-DSS}
There exist $\theta\in(0,1/4)$ and $C<\infty$ such that the following
holds. Let $U\in C(B_2)$ solve
\begin{equation}\label{eq:diffuse-alternative-equation}
F(D^2U)
=
\frac\gamma\eta
\beta\left(\frac{U+\sigma}{\eta}\right)
\qquad\text{in }B_2,
\end{equation}
where
$$
U(0)=0,
\qquad
0<\gamma\leq1,
\qquad
\eta>0,
\qquad
|\sigma|\leq c_0\eta.
$$
Put
$$
A:=\norm{U}_{L^\infty(B_1)}.
$$
Then at least one of the following alternatives holds:
\begin{align}
\theta^{-1}\norm{U}_{L^\infty(B_\theta)}
&\leq
\frac12A+C\sqrt\gamma,
\label{eq:diffuse-decay}
\\
\norm{\nabla U}_{L^\infty(B_\theta)}
&\leq
C\left(A+\sqrt\gamma\right).
\label{eq:diffuse-lipschitz}
\end{align}
The constants are independent of $\eta$.
\end{proposition}

\begin{proof}
Assume, toward a contradiction, that no universal constant works. We can
then choose a sequence for which both alternatives fail with $C$ replaced
by $j$. Write
$$
A_j:=\norm{U_j}_{L^\infty(B_1)},
\qquad
H_j:=A_j+\sqrt{\gamma_j},
\qquad
u_j:=\frac{U_j}{H_j}.
$$
The rescaled operators
$$
\widetilde F_j(M)
:=
\frac1{H_j}F_j(H_jM)
$$
have the same ellipticity constants, and $u_j$ satisfies an equation of the
form \eqref{eq:bounded-diffuse-sequence} with
$$
\mu_j:=\frac{\gamma_j}{H_j^2},
\qquad
\widetilde\eta_j:=\frac{\eta_j}{H_j},
\qquad
\widetilde\sigma_j:=\frac{\sigma_j}{H_j}.
$$
Failure of the decay alternative gives
$$
\theta^{-1}\norm{u_j}_{L^\infty(B_\theta)}
>
\frac12\frac{A_j}{H_j}+j\sqrt{\mu_j}.
$$
Since the left-hand side is bounded by $\theta^{-1}$, it follows that
\begin{equation}\label{eq:normalized-strength-vanishes}
\sqrt{\mu_j}\leq\frac1{j\theta},
\qquad
\frac{A_j}{H_j}=1-\sqrt{\mu_j}\longrightarrow1.
\end{equation}
In particular, \Cref{lem:bounded-diffuse-compactness} applies. After
passing to a subsequence,
$$
u_j\longrightarrow u_\infty
\qquad\text{locally uniformly in }B_1,
$$
and
\begin{equation}\label{eq:normalized-homogeneous-limit}
F_\infty(D^2u_\infty)=0
\qquad\text{in }B_1.
\end{equation}
Moreover,
$$
u_\infty(0)=0,
\qquad
\norm{u_\infty}_{L^\infty(B_1)}\leq1,
$$
and the failed decay inequality ensures that $u_\infty$ is nontrivial in
$B_\theta$.

Let $p:=\nabla u_\infty(0)$. By
\eqref{eq:homogeneous-affine-approximation}, after decreasing $\theta$ if
necessary,
\begin{equation}\label{eq:homogeneous-limit-affine-trapping}
\norm{u_\infty-p\cdot x}_{L^\infty(B_{2\theta})}
\leq
C\theta^{1+\bar\alpha}.
\end{equation}
If $|p|$ is smaller than a fixed universal number $a_*>0$, chosen after
$\theta$, then \eqref{eq:homogeneous-limit-affine-trapping} gives
$$
\theta^{-1}\norm{u_\infty}_{L^\infty(B_\theta)}
\leq\frac14.
$$
Local uniform convergence and $A_j/H_j\to1$ then contradict the failed
decay inequality. We must therefore have $|p|\geq a_*$.

Define
$$
v_j(x)
:=
\frac{u_j(2\theta x)}{2\theta},
\qquad x\in B_2.
$$
For $\theta$ fixed sufficiently small, and then $j$ sufficiently large,
\eqref{eq:homogeneous-limit-affine-trapping} gives
$$
\norm{v_j-p\cdot x}_{L^\infty(B_1)}
\leq\varepsilon_*.
$$
The spatial-amplitude scaling leaves the effective strength $\mu_j$
unchanged and preserves the shifted-layer bound. Hence
\Cref{lem:diffuse-affine-continuation} applies and yields
$$
\norm{\nabla u_j}_{L^\infty(B_\theta)}\leq C.
$$
Scaling back by $H_j$ gives
$$
\norm{\nabla U_j}_{L^\infty(B_\theta)}
\leq
C\left(A_j+\sqrt{\gamma_j}\right),
$$
contradicting the assumed failure of
\eqref{eq:diffuse-lipschitz} with constant $j$. This proves the
proposition.
\end{proof}

\subsection{Linear growth from a transition point}

The preceding alternative immediately gives the scale-invariant growth
estimate needed in the final blow-up argument.

\begin{corollary}[Growth from a transition point]
\label{cor:growth}
Let $U$ solve \eqref{eq:diffuse-alternative-equation} in $B_2$, with
$U(0)=0$ and $0<\gamma\leq1$. Then
\begin{equation}\label{eq:growth}
|U(x)|
\leq
C\left(
\norm{U}_{L^\infty(B_1)}+\sqrt\gamma
\right)|x|
\qquad\text{for }x\in B_{1/2}.
\end{equation}
\end{corollary}

\begin{proof}
Set
$$
a_k
:=
\theta^{-k}
\norm{U}_{L^\infty(B_{\theta^k})},
\qquad k\geq0.
$$
Apply \Cref{prop:diffuse-DSS} to the natural rescaling
$$
U_k(x):=\theta^{-k}U(\theta^kx).
$$
The effective strength remains $\gamma$, while the thickness and the shift
are divided by $\theta^k$; in particular, the shifted-layer bound is
preserved. If the decay branch occurs at scale $k$, then
$$
a_{k+1}
\leq
\frac12a_k+C\sqrt\gamma.
$$
If only this branch occurs, iteration gives
$$
\sup_{k\geq0}a_k
\leq
C(a_0+\sqrt\gamma).
$$
If the Lipschitz branch first occurs at an index $k_0$, then
$U_{k_0}(0)=0$ and \eqref{eq:diffuse-lipschitz} controls all smaller radii:
$$
a_k
\leq
C(a_{k_0}+\sqrt\gamma)
\qquad\text{for }k>k_0.
$$
The preceding decay steps bound $a_{k_0}$ by
$C(a_0+\sqrt\gamma)$. Thus the same bound for $\sup_k a_k$ holds in both
cases. Comparing an arbitrary radius with two consecutive powers of
$\theta$ proves \eqref{eq:growth}.
\end{proof}

\section{Transition-boundary growth and completion of the proof}
\label{sec:completion}

The local estimates of \Cref{sec:normalization-local} reduce the normalized
theorem to the control of the boundary-to-domain seminorm $\mathcal B(v)$.
The growth estimate obtained in \Cref{cor:growth} supplies precisely this
missing bound. The weight in \eqref{eq:Bint} is important here: it compensates
for the amplitude introduced when a level-boundary point approaches the
artificial boundary of $B_{2/3}$.

\begin{proposition}[Transition-boundary Lipschitz bound]
\label{prop:Bbound}

Every solution of \eqref{eq:normalized} satisfies
\begin{equation}\label{eq:Bbound}
\mathcal B(v)\leq C.
\end{equation}
\end{proposition}

\begin{proof}
Fix
$$
z\in\Gamma_\delta(v)\cap B_{2/3}
$$
and set
\begin{equation}\label{eq:boundary-distance-scale}
d_z:=\frac23-|z|,
\qquad
\rho_z:=\frac{d_z}{8}.
\end{equation}
Since $z$ is an interior point of the relative level boundary and $v$ is
continuous,
\begin{equation}\label{eq:level-boundary-value}
|v(z)|=c_0\delta.
\end{equation}
Moreover,
$$
|z|+2\rho_z
=
|z|+\frac14\left(\frac23-|z|\right)
\leq\frac23,
$$
so that $B_{2\rho_z}(z)\subset B_{2/3}\subset B_1$.

Consider the natural spatial-amplitude rescaling
\begin{equation}\label{eq:boundary-point-rescaling}
U(y)
:=
\frac{v(z+\rho_zy)-v(z)}{\rho_z},
\qquad y\in B_2,
\end{equation}
and define
$$
F_z(M)
:=
\rho_zF\left(\frac{M}{\rho_z}\right),
\qquad
\eta_z:=\frac{\delta}{\rho_z},
\qquad
\sigma_z:=\frac{v(z)}{\rho_z}.
$$
The operator $F_z$ is normalized and has the same ellipticity constants as
$F$. By the scaling identity \eqref{eq:rescaled-singular-equation}, the
function $U$ satisfies
\begin{equation}\label{eq:boundary-point-equation}
F_z(D^2U)
=
\frac{\gamma}{\eta_z}
\beta\left(\frac{U+\sigma_z}{\eta_z}\right)
\qquad\text{in }B_2.
\end{equation}
Furthermore,
$$
U(0)=0,
\qquad
|\sigma_z|=c_0\eta_z,
\qquad
\norm{U}_{L^\infty(B_1)}
\leq
\frac{2}{\rho_z}.
$$
Thus all the hypotheses of \Cref{cor:growth} are satisfied.

Let $x\in B_{3/4}\setminus\{z\}$. We distinguish two ranges. If
$|x-z|\leq\rho_z/2$, put $y=(x-z)/\rho_z\in B_{1/2}$. By
\eqref{eq:boundary-point-rescaling} and \Cref{cor:growth},
\begin{align*}
d_z\frac{|v(x)-v(z)|}{|x-z|}
&=
d_z\frac{|U(y)|}{|y|}
\\
&\leq
Cd_z\left(
\norm{U}_{L^\infty(B_1)}+\sqrt\gamma
\right)
\\
&\leq
Cd_z\left(\frac2{\rho_z}+1\right)
\leq C,
\end{align*}
where we used $\gamma\leq1$ and $d_z=8\rho_z$. This is the precise
cancellation for which the weight in \eqref{eq:Bint} was introduced.

If instead $|x-z|>\rho_z/2$, the normalization
$\norm{v}_{L^\infty(B_1)}\leq1$ gives
$$
d_z\frac{|v(x)-v(z)|}{|x-z|}
\leq
2d_z\frac{2}{\rho_z}
=32.
$$
The two estimates are uniform in $z$ and $x$. Taking the suprema in
\eqref{eq:Bint} proves \eqref{eq:Bbound}.
\end{proof}

\begin{proof}[Proof of \Cref{thm:normalized}]
By \Cref{cor:reduction} and \Cref{prop:Bbound},
$$
\norm{\nabla v}_{L^\infty(B_{1/2})}
\leq
C\left(1+\mathcal B(v)\right)
\leq C.
$$
This is \eqref{eq:normalized-bound}.
\end{proof}

\begin{proof}[Proof of \Cref{thm:main-intro}]
The result follows from \Cref{prop:normalization}.
\end{proof}

\section{\texorpdfstring{Sharpness of the additive $\sqrt{\alpha}$ term}%
{Sharpness of the square-root additive contribution}}
\label{sec:sharpness}

The scale $\sqrt\alpha$ in \Cref{thm:main-intro} is dictated by the
reaction layer and is quantitatively optimal. We demonstrate this with an
explicit radial family for the Laplacian. More precisely, when $n\geq2$ we
construct solutions whose $L^\infty$ norms are $o(\sqrt\alpha)$, while their
gradients remain bounded below by a fixed multiple of $\sqrt\alpha$.

\lealemma{label=lemRadialReactionProfileExistence, uses={SingularPerturbationProblemData}, context={Treat global existence and smoothness for this radial semilinear ODE, including regularity at the origin, as an imported standard ODE result. The chosen value a lies in (-1,1) and satisfies beta(a) > 0.}}{Assume that $\beta\not\equiv0$, and choose $a\in(-1,1)$ such that
$\beta(a)>0$. Let $\phi$ be the smooth radial solution of
\begin{equation}\label{eq:radial-profile-ode}
\phi''(s)
+
\frac{n-1}{s}\phi'(s)
=
\beta(\phi(s)),
\qquad
\phi(0)=a,
\qquad
\phi'(0)=0.
\end{equation}
The equation at the origin is understood in the usual smooth radial sense.
In particular,
\begin{equation}\label{eq:radial-second-derivative}
\phi''(0)=\frac{\beta(a)}{n}>0.
\end{equation}
Consequently, there exist constants $s_*>0$ and $c_*>0$, depending only on
$n$, $\beta$, and the choice of $a$, such that
\begin{equation}\label{eq:radial-positive-slope}
\phi'(s_*)=c_*.
\end{equation}
}

\lealemma{label=lemRadialReactionProfileGrowth, uses={lemRadialReactionProfileExistence,SingularPerturbationProblemData}}{We first record the global growth of the profile. Multiplying
\eqref{eq:radial-profile-ode} by $s^{n-1}$ gives
\begin{equation}\label{eq:radial-monotonicity}
s^{n-1}\phi'(s)
=
\int_0^s t^{n-1}\beta(\phi(t))\,dt
\geq0.
\end{equation}
Thus $\phi$ is nondecreasing. If it never reaches the upper endpoint of the
support of $\beta$, then it is bounded. Otherwise, once it has passed that
endpoint, the reaction vanishes permanently and $\phi$ satisfies the radial
harmonic equation
$$
\phi''+\frac{n-1}{s}\phi'=0.
$$
It follows that
\begin{equation}\label{eq:radial-profile-growth}
|\phi(s)|
\leq
\begin{cases}
C, & n\geq3,\\[2mm]
C\bigl(1+\log(2+s)\bigr), & n=2,
\end{cases}
\qquad s\geq0,
\end{equation}
where $C$ depends only on $n$, $\beta$, and $a$.
}

\leadefinition{label=RadialSharpnessFamily, uses={lemRadialReactionProfileExistence}}{For $\eps,\alpha>0$, define
\begin{equation}\label{eq:radial-sharpness-family}
u_{\eps,\alpha}(x)
:=
\eps\,
\phi\left(\frac{\sqrt\alpha}{\eps}|x|\right).
\end{equation}
}%
\lealemma{label=lemRadialSharpnessEquationAndSlope, uses={RadialSharpnessFamily,lemRadialReactionProfileExistence,SingularPerturbationProblemData}}{A direct computation using \eqref{eq:radial-profile-ode} gives
\begin{equation}\label{eq:radial-sharpness-equation}
\Delta u_{\eps,\alpha}
=
\frac\alpha\eps
\beta\left(\frac{u_{\eps,\alpha}}\eps\right)
\qquad\text{in }\mathbb R^n.
\end{equation}
At the point
$$
x_{\eps,\alpha}
:=
\frac{\eps s_*}{\sqrt\alpha}e_1,
$$
we have, by \eqref{eq:radial-positive-slope},
\begin{equation}\label{eq:radial-gradient-lower-bound}
|\nabla u_{\eps,\alpha}(x_{\eps,\alpha})|
=
c_*\sqrt\alpha.
\end{equation}
Whenever $\eps s_*/\sqrt\alpha<1/2$, this point belongs to $B_{1/2}$.
}

\lealemma{label=lemRadialSharpnessAmplitude, uses={RadialSharpnessFamily,lemRadialReactionProfileGrowth}}{On the other hand, \eqref{eq:radial-profile-growth} yields
\begin{equation}\label{eq:radial-Linfty-bound}
\norm{u_{\eps,\alpha}}_{L^\infty(B_1)}
\leq
\begin{cases}
C\eps, & n\geq3,\\[2mm]
C\eps\left(
1+\log\left(2+\dfrac{\sqrt\alpha}{\eps}\right)
\right), & n=2.
\end{cases}
\end{equation}
To make the optimality explicit, let $\alpha\downarrow0$ and choose
$\eps=\alpha$. Then
$$
x_{\alpha,\alpha}
=
s_*\sqrt\alpha\,e_1
\in B_{1/2}
$$
for all sufficiently small $\alpha$, and
\begin{equation}\label{eq:sharpness-gradient-final}
\norm{\nabla u_{\alpha,\alpha}}_{L^\infty(B_{1/2})}
\geq
c_*\sqrt\alpha.
\end{equation}
At the same time,
\begin{equation}\label{eq:sharpness-amplitude-final}
\frac{
\norm{u_{\alpha,\alpha}}_{L^\infty(B_1)}
}{\sqrt\alpha}
\longrightarrow0.
\end{equation}
Indeed, the ratio in \eqref{eq:sharpness-amplitude-final} is bounded by
$C\sqrt\alpha$ when $n\geq3$, and by
$$
C\sqrt\alpha
\left(
1+\log\left(2+\frac1{\sqrt\alpha}\right)
\right)
$$
when $n=2$; both quantities tend to zero.
}

\leaproposition{label=propSharpnessSqrtAlpha, uses={MainLipschitzEstimate,lemRadialSharpnessEquationAndSlope,lemRadialSharpnessAmplitude}, context={The quantified negation is over constants C having only the universal dependence allowed in MainLipschitzEstimate and over functions omega satisfying omega(alpha)/sqrt(alpha) -> 0 as alpha -> 0. Interpret gradient norms as Lipschitz bounds after the smooth radial construction has supplied classical gradients.}}{It follows that the additive term in \Cref{thm:main-intro} cannot be
replaced by any quantity of smaller order. More precisely, there is no
universal estimate of the form
$$
\norm{\nabla u_\eps}_{L^\infty(B_{1/2})}
\leq
C\left(
\norm{u_\eps}_{L^\infty(B_1)}+\omega(\alpha)
\right)
$$
with
$$
\omega(\alpha)=o(\sqrt\alpha)
\qquad\text{as }\alpha\downarrow0.
$$
Thus the $\sqrt\alpha$ contribution in the main estimate is sharp.
}

\end{document}